\documentclass[12pt,reqno]{amsart}

\usepackage{amsmath}
\usepackage{amssymb}
\usepackage{array}
\usepackage{enumitem}
\usepackage{verbatim} 
\usepackage{hyperref}

\usepackage{pstricks-add}

\input cyracc.def 
\newfam\cyrfam

\theoremstyle{plain}
\newtheorem{theorem}{Theorem}[section]
\newtheorem{lemma}[theorem]{Lemma}
\newtheorem{proposition}[theorem]{Proposition}
\newtheorem{corollary}[theorem]{Corollary}

\newtheorem{definition}[theorem]{Definition}
\theoremstyle{remark}
\newtheorem{remark}{Remark}[section]
\newtheorem{example}{Example}[section]

\newtheorem*{acknowledgment}{Acknowledgment}

\numberwithin{equation}{section}

\newcommand{\bK}{\mathbb{K}}
\newcommand{\bP}{\mathbb{P}}

\newcommand{\cX}{\mathcal{X}}

\newcommand{\cP}{\mathcal{P}}

\newcommand{\cR}{\mathcal{R}}
\newcommand{\cI}{\mathcal{I}}
\newcommand{\cE}{\mathcal{E}} 

\newcommand{\bfa}{\mathbf{a}}
\newcommand{\bfb}{\mathbf{b}}

\newcommand{\Aut}{\mathrm{Aut}}

\newcommand{\Bij}{\mathrm{Bij}}
\newcommand{\Map}{\mathrm{Map}}

\newcommand{\im}{\mathrm{im}}

\newcommand{\id}{\mathrm{id}}
\newcommand{\pr}{\mathrm{pr}}

\newcommand{\dom}{\mathrm{dom}}
\newcommand{\loc}{\mathrm{loc}}
\newcommand{\book}{\mathrm{bookk}}

\newcommand{\Bigsetof}[2]{\begin{Bmatrix} #1 \,\Big|\, #2 \end{Bmatrix}}

\newcommand{\inv}{^{-1}}

\newcommand{\ssk}{\smallskip}
\newcommand{\nin}{\noindent}

\newcommand{\All}{\mathbf{O}}
\newcommand{\Dia}{\mathbf{I}}
\newcommand{\Gr}{\mathrm{Graph}} 

\begin{document}

\title[Universal Associative  Geometry]{Universal Associative Geometry}

\author{Wolfgang Bertram}

\address{Institut \'{E}lie Cartan de Lorraine (IECL) \\
Universit\'{e} de Lorraine, CNRS, INRIA, \\
Boulevard des Aiguillettes, B.P. 239, \\
F-54506 Vand\oe{}uvre-l\`{e}s-Nancy, France}

\email{\url{wolfgang.bertram@univ-lorraine.fr}}

\subjclass[2010]{
08A02  	
16Y30  	
18B10  	
20L05  	
58H05  	
20N10  
17C37  
}

\keywords{
associoid, bisection, (pre)groupoid, principal equivalence relation (principal bundle), (semi)torsor, binary relation, 
relational composition, universal algebra
}

\begin{abstract}
We generalize parts of the theory of {\em associative geometries} developed in
\cite{BeKi10, BeKi10b, Be12} in the framework of universal algebra:
we prove that certain {\em associoid} structures, such as {\em pregroupoids} and 
{\em principal equivalence relations},  have a natural prolongation from a set $\Omega$
to the power set o $\Omega$.
We reinvestigate  the case of {\em homogeneous pregroupoids} 
(corresponding to the {\em projective geometry of a group}, see \cite{Be12}) from the point of
view of {\em pairs of commuting principal equivalence relations}.
We use  the ternary approach to groupoids developed by Kock \cite{Ko05, Ko07}, and 
the torsors defined by our construction  can be seen
as a generalisation of the known {\em groups of bisections of a groupoid}. 
\end{abstract}

\maketitle

\setcounter{tocdepth}{1}
\tableofcontents

\section*{Introduction}

\subsection{Associativity and universal algebra.}
The associative law plays a special r\^ole in universal algebra:
on the one hand, associative structures such as groups, groupoids, semigroups  and associative algebras are {\em objects}  
studied by means of  universal algebra, on the same level as, e.g., Lie- or Jordan algebraic structures (see \cite{Cohn});
but on the other hand, the associative law is  foundational  for the topic of  universal algebra itself --
coming in via {\em composition} of mappings and  of binary relations, via {\em lattices} and {\em set-theory
 (intersections and unions)}.
Thus associativity belongs both to the ``input'' and to the ``output'' side of  universal algebra.

\ssk
In a likewise way, when developing with M.\ Kinyon the theory of {\em associative geometries} (\cite{BeKi10, BeKi10b,
BeKi12}), our first aim
was to define a geometric {\em object} (``associative
coquecigrue'') corresponding to associative algebras in a similar way as Lie groups correspond
to Lie algebras. However, in subsequent work \cite{Be12}, it became more and more visible that much of this
approach really belongs to {\em methods of universal algebra}, and thus
there ought to be a {\em universal  geometric algebra of associativity}. 
The aim of the present work is to explain what I mean by
this. In a nutshell, it is a {\em  geometric language to describe the duality between the power set $\cP(\Omega)$ of
a set  $\Omega$
and the set $\cE(\Omega)$ of equivalence relations on $\Omega$, in analogy with the duality between a projective
space $\bP(W)$ and its dual projective space $\bP(W^*)$}.
Indeed, when coming accross  the following phrase (\cite{Cohn}, 
p. 16): ``the quotient sets of a given set $A$ are to some extent dual to the subsets of $A$, but this
duality is not complete'', I had the feeling that this remark  is an important guidepost,
and that it should be worth understanding  where it leads, when following the path to its end. 

\subsection{Terminology} 
Before explaining our results, it will be useful to fix some terminology (see 
Appendix \ref{APP} for formal definitions).
Motivated by \cite{BeKi10, BeKi10b}, and by work of Kock \cite{Ko82, Ko87, Ko05, Ko07}, we work with
{\em ternary} and possibly {\em partially defined} product maps (which we often denote
by $(xyz)$ or $[xyz]$ and Kock denotes in most papers
by $xy\inv z$), so the most general structure of this kind is a {\em semi-associoid}, a set $M$ with a partially defined
product map satisfying the 
{\em identity of para-associativity}
\[
(xy(zuv))= (x(uzy)v)=((xyz)uv)\,   \tag{PA}
\]
and which  is called an {\em associoid} if it satisfies, moreover,  the 
\emph{idempotent law} 
\[
(xxy)=y, \quad (wzz)=w  \, . \tag{IP}
\]
According to the nature of the domain $D \subset M^3$ of definition of the product map, the following more or less
classical associative objects are defined, summarized by the following diagram
of structures (categories, in fact) and their inclusions\footnote{There is a serious problem of terminology: 
for most of these objects, there is no universally accepted term, und we are well-aware that the terminology we use
here may be unsatisfying. We hope that, in a not too distant future, the mathematical community will agree on a better terminology, once
the importance of ternary product structures is more widely recognized.}: 
$$
\begin{matrix}
\mbox{\em  semi-associoid } & \supset & \mbox{\em associoid }
\cr
 \cup & & \cup 
 \cr
\mbox{\em semi-pregroupoid } & \supset & \mbox{\em pregroupoid }
\cr
 & & \cup 
 \cr
\cup  &  & \mbox{\em (left or right) principal equivalence relation (prev) }
\cr
   & & \cup 
 \cr
\mbox{\em semi-torsor (semi-pregroup) } & \supset & \mbox{\em torsor (pregroup) }
\end{matrix}
$$
Objects in the last line have {\em everywhere defined} products: $D = M^3$. 
A {\em group} is nothing but a torsor with some fixed base point $y$ (unit element) and product
$xz = (xyz)$, and a {\em groupoid}
is nothing but a pregroupoid with some fixed bisection (set of units).
The notion of {\em (left or right) prev} is an abstract-algebraic version of the one of {\em principal bundle}, 
stripped off the usual topological  conditions. 
The domain $D$ for (semi-)pregroupoids is defined in terms of {\em two equivalence relations $a,b$}, the fibers of
the two projections ``target'' and ``domain'':
\begin{equation}
\begin{matrix} M/a & \leftarrow & M & \rightarrow & M/b \end{matrix} \, ,
\end{equation}
whereas for left or right prev's, just one of the two projections suffices to define $D$.

\subsection{First stage: naked sets, and
composition of  binary relations}
We fix a ground set $\Omega$ throughout this work.
{\em Duality} between the power set $\cX:=\cP(\Omega)$ and the set of equivalence relations $\cX':=\cE(\Omega)$ 
means that $\cX'$ ``structurizes'' or ``coordinizes'' $\cX$, and vice versa -- but this duality is not complete!
For instance, 
\begin{itemize}
\item
there is a natural {\em remoteness} or {\em transversality} relation for pairs
$(x,a) \in \cP(\Omega) \times \cE(\Omega)$ and for pairs
$(a,b) \in \cE(\Omega) \times \cE(\Omega)$ -- but not for pairs $(x,y)$ of subsets,
\item
we may {\em compose relations} (relational composition) -- but we cannot  ``compose sets''. 
\end{itemize}

\nin
Our first result (Theorem \ref{th:semitorsor1}) says that the second claim is not quite true: we can ``compose sets'', 
{\em provided given a fixed pair $(a,b)$ of commuting equivalence relations}.
Statement and proof of this result are elementary and go back to the very definition of relational composition:
start by observing that usual relational composition $x\circ y\inv \circ z$ of subsets $x,y,z \subset \Omega$ 
would be well-defined if there were given a direct product structure $\Omega = \Omega_1 \times \Omega_2$, but it would of course
depend on this direct product structure.
Next, observe that a direct product structure on $\Omega$ can be described by
{\em two commuting transversal equivalence relations $a,b$}, namely the fibers of the two projections $\pr_1, \pr_2$
(Lemma \ref{la:transv2}). Then rewrite the definition of relational composition $xy\inv z$ in terms of $x,a,y,b,z$; you get
\begin{equation}\label{eqn:0.1} 
x y\inv z =
(xyz)_{ab}^\book: =
\Bigsetof{\omega \in \Omega}
{\begin{array}{c}
\exists \xi \in x,
\exists \eta \in y,
\exists \zeta \in z : \\
\omega \sim_a \zeta, \quad
\eta \sim_a \xi,\quad
\omega \sim_b \xi, \quad
\eta \sim_b \zeta 
\end{array}} .
\end{equation}
It is truly remarkable that the  same formula still defines a para-associative ternary composition
if $a,b$ commute but  are no longer transversal, nor even everywhere defined:
it suffices that $(a,b)$ be a {\em pair of commuting equivalence relations ``in'' $\Omega$} to obtain
a semitorsor structure on $\cP(\Omega)$, which we call  the {\em book-keeping semitorsor}
because this structure underlies the stronger structures to be defined later.
The definition of the set $(xyz)_{ab}^\book$ may be described by the following ``parallelogram figure''

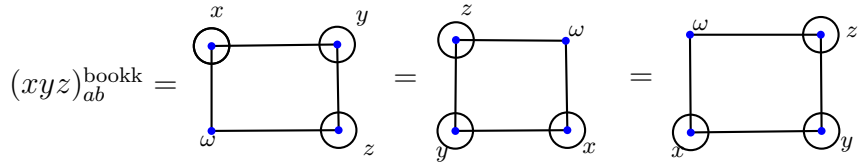
\begin{center}
\psset{xunit=0.5cm,yunit=0.5cm,algebraic=true,dotstyle=o,dotsize=3pt 0,linewidth=0.8pt,arrowsize=3pt 2,arrowinset=0.25}
\begin{pspicture*}(-5.3,-2.00)(19.46,2.5)
\psline(0.36,1.12)(3.7,1.16)
\psline(0.36,1.12)(0.36,-1.14)
\psline(0.36,-1.14)(3.74,-1.12)
\psline(3.7,1.16)(3.74,-1.12)
\pscircle(3.7,1.16){0.25}
\pscircle(3.74,-1.12){0.25}
\pscircle(0.36,1.12){0.25}
\pscircle(0.36,1.12){0.25}
\rput[tl](-5.02,0.56){$ (xyz)_{ab}^\book =$}
\psline(6.86,1.28)(6.84,-1.14)
\psline(6.84,-1.14)(9.82,-1.12)
\psline(9.82,-1.12)(9.78,1.26)
\psline(6.86,1.28)(9.78,1.26)
\pscircle(6.86,1.28){0.25}
\pscircle(6.84,-1.14){0.25}
\pscircle(9.82,-1.12){0.25}
\psline(13.08,1.42)(16.56,1.42)
\psline(16.56,1.42)(16.58,-1.14)
\psline(13.08,1.42)(13.1,-1.18)
\psline(13.1,-1.18)(16.58,-1.14)
\pscircle(16.56,1.42){0.25}
\pscircle(16.58,-1.14){0.25}
\pscircle(13.1,-1.18){0.25}
\rput[tl](5.2,0.36){$=$}
\rput[tl](11.46,0.38){$=$}
\begin{scriptsize}
\psdots[dotstyle=*,linecolor=blue](0.36,1.12)
\psdots[dotstyle=*,linecolor=blue](3.7,1.16)
\psdots[dotstyle=*,linecolor=blue](0.36,-1.14)
\rput[bl](0.06,-1.54){{$\omega$}}
\psdots[dotstyle=*,linecolor=blue](3.74,-1.12)
\rput[bl](4.2,1.66){$y$}
\rput[bl](4.38,-1.78){$z$}
\rput[bl](0.32,1.88){$x$}
\psdots[dotstyle=*,linecolor=blue](6.86,1.28)
\psdots[dotstyle=*,linecolor=blue](6.84,-1.14)
\psdots[dotstyle=*,linecolor=blue](9.82,-1.12)
\psdots[dotstyle=*,linecolor=blue](9.78,1.26)
\rput[bl](9.86,1.38){{$\omega$}}
\rput[bl](6.98,1.94){$z$}
\rput[bl](6.32,-1.96){$y$}
\rput[bl](10.22,-1.72){$x$}
\psdots[dotstyle=*,linecolor=blue](13.08,1.42)
\rput[bl](13.16,1.54){{$\omega$}}
\psdots[dotstyle=*,linecolor=blue](16.56,1.42)
\psdots[dotstyle=*,linecolor=blue](16.58,-1.14)
\psdots[dotstyle=*,linecolor=blue](13.1,-1.18)
\rput[bl](17.22,1.36){$z$}
\rput[bl](17.08,-1.7){$y$}
\rput[bl](12.58,-1.84){$x$}
\end{scriptsize}
\end{pspicture*}
\end{center}

\nin where sets $x,y,z$ are indicated by circles, elements $\xi,\eta,\zeta$ by solid points inside these circles;
 a horizontal link indicates that the points are in relation $a$, and a vertical link that they are in relation $b$. 
This kind of presentation and of 
results  can be extended to general $n$-ary composition, and in particular to binary composition;
and it should also be possible to generalize it to a general categorical framework (``allegories'', see
\cite{FS90} or  \cite{Jo02}  Section A3). We will not do this here, but we try to use from the very beginning  assumptions and
notations that are adapted to possible ``allegorical'' generalizations.

\subsection{Second stage: pregroupoids and their (local) bisections}
Now, our naked set will get dressed:
we assume that its underwear is a semi-pregroupoid structure $(\Omega,a,b,[\quad])$.
Then the preceding result generalizes in the following way (Theorem \ref{th:semitorsor2}):
{\em the power set $\cP(\Omega)$ becomes a semitorsor when equipped with the everywhere defined
ternary law given by
\begin{equation}\label{eqn:0.3}
(xyz)_{ab}  :=
\Bigsetof{\omega \in \Omega}
{\begin{array}{c}
\exists \xi \in x,
\exists \eta \in y,
\exists \zeta \in z : \\
\eta \sim_a \xi,\quad
\eta \sim_b \zeta ,
\quad \omega = [\xi \eta \zeta] 
\end{array}}  ,
\end{equation}
and, moreover,}

\begin{enumerate}
\item
{\em with its natural equivalence relations $\bfa,\bfb$ induced by $a,b$, $(\cP(\Omega),(\quad)_{ab})$
becomes itself a semi-pregroupoid,
\item
the {\em space $U_{ab}^\loc$ of local bisections of $(a,b)$} is a sub semi-pregroupoid of the preceding.}
\end{enumerate}
 
\nin
Now
let us dress up our set $\Omega$ a bit more: assume it is a pregroupoid, that is, it satisfies moreover the idempotent law.
This will not imply that $(\cP(\Omega),(\quad)_{ab})$ satisfies the idempotent law, but  we have
the following (Theorem \ref{th:torsor1}):
{\em the set $U_{ab}^\loc$ of local bisections becomes a pregroupoid, and the set $U_{ab}$ of global bisections becomes
a torsor, as summarized by the following diagram:}
$$
\begin{matrix}
\mbox{\em semi-pregroupoid } (\cP(\Omega),\bfa,\bfb,(\quad)_{ab}) & \supset & \mbox{\em pregroupoid } 
(U_{ab}^\loc,\bfa,\bfb,(\quad)_{ab})
\cr
\cup  & & \cup 
 \cr
\mbox{\em semi-torsor } (\cP(\Omega),(\quad)_{ab}) & \supset & \mbox{\em torsor  } (U_{ab},(\quad)_{ab})
\end{matrix}
$$
\nin
These torsor and semi-torsor structures generalize some more classical objects:
if the pregroupoid $\Omega$ is in fact a groupoid, then, as stated in \cite{CW}, p.\ 106, {\em  the set $2^\Omega$
carries a semigroup structure} (corresponding to our semigroup $(x,z)\mapsto (xyz)_{ab}$, where $y$ is the bisection
of units), {\em having several interesting sub-semigroups}, among them the sub-semigroup of local bisections and the
group of bisections. 
The case of two transversal equivalence relations ($a \top b$) leads right back to the case of a naked set with
ordinary composition: indeed, the {\em pair pregroupoid} on $(E ,F)= (\Omega/a,\Omega/b)$ is given by purely set-theoretic data,
and the diagram above reduces to
$$
\begin{matrix}
\mbox{\em semi-pregroupoid } (\cR(E,F),\bfa,\bfb)  & \supset & \mbox{\em pregroupoid of local bijections } E \to F 
\cr
\cup  & & \cup 
 \cr
\mbox{\em semi-torsor } \cR(E,F) & \supset & \mbox{\em torsor of global bijections } E \to F
\end{matrix} 
$$
\nin where the space $\cR(E,F)$ of binary relations is equipped with its ``usual'' ternary product
$x y\inv z$.
Thus one may say that pregroupoid-like structures on $\Omega$ extend, in a canonical way,
to structures of the same kind on $\cP(\Omega)$.  Of course, this offers the possibility to repeat such constructions
on the level of higher order structures $\cP(\cP(\Omega))$, and so on.

\subsection{Third stage: commuting  principal equivalence relations} 
We continue to dress up our set by assuming that $(\Omega,a,b,[\quad])$ is a {\em pair of commuting
left and right prev's}.
Essentially, this means that the partially defined ternary map $[\quad]$ now lives on the union of
$a \times \Omega)$ and  $\Omega\times b$, and not only on their intersection, as for pregroupoids.
This has some nice additional consequences (Section \ref{sec:distrib}):
first of all, not only the space $U_{ab}$ of bisections, but also the spaces
$U_a,U_b$ of sections of $a$, resp.\ of $b$, carry canonical group structures (which we denote by $+$,
although they need not be commutative), and $U_{ab}$ acts on these groups from both sides,
such that {\em one-sided distributivity laws} hold. 
The whole structure thus resembles strongly to what one calls a {\em near-ring} (see \cite{Pi77}).
It is, then, possible to describe this object by quite explicit formulae (involving a {\em canonical kernel},
generalizing  the {\em Bergman operators}, which in turn define
the {\em Bergman kernels} appearing in the theory of bounded symmetric domains),
making it amenable to methods borrowed from the theory of associative or Jordan algebras. 

\subsection{Fourth stage: homogeneous pregroupoids}
This is the royal dress:
$\Omega$ is assumed to be a group, in which we fix two subgroups $A$ and $B$, inducing two commuting
left and right prev's $a$ and $b$ whose fibers are right cosets of $A$ and left  cosets of $B$:
\begin{equation}
\begin{matrix}
\Omega /B & \leftarrow & \Omega & \rightarrow & A\backslash \Omega \end{matrix}
\end{equation}
In this case, the results described above are precisely those presented in \cite{Be12}:
writing $\Omega$ additively (but not assumed to be commutative), the ternary law is given by
$[\xi \eta \zeta ]=\xi -\eta + \zeta$, and so (\ref{eqn:0.3}) becomes
\begin{equation}\label{eqn:str}
(xyz)_{ab} = 
\Bigsetof{\omega \in \Omega}
{\begin{array}{c}
\exists \xi \in x,
\exists \eta \in y,
\exists \zeta \in z,\exists \alpha \in A, \exists \beta \in B : \\
\eta = \alpha + \xi, \quad
\eta = \zeta + \beta , \quad
\omega = \xi - \eta + \zeta 
\end{array}} 
\end{equation}
The set $(xyz)_{ab}$ has been denoted by $\Gamma(x,a,y,b,z)$, and the
 three equations appearing here have been called the {\em structure equations} in \cite{Be12}.
Comparison with loc.\ cit.\  shows that many proofs  are now greatly  simplified, due to the 
conceptual approach, but, on the other hand, some results are specific to the group case, namely,
the presence of additional symmetries or the possibility of certain anti-symmetries  -- these items
are very important, and we are going to add some comments on this.

\subsection{Further topics; duality}
Once the viewpoint of groupoids and ternary products (pre-groupoids, prev's) is adopted, the
results presented here are completely natural, or (what is more or less  the same)  ``trivial'': 
they are just an unfolding of logical consequences of  definitions.
However, taking this impression to be the whole story  would mean to miss the main point of the present approach.
Indeed, it seems as if there were some ``quantization effect'':
the viewpoint presented here is ``classical'' and fits completely with general, possibly non-transitive, 
pregroupoids; however, the more the pregroupoid becomes homogeneous, by a sort of 
``quantization effect'', the structure becomes richer and more rigid due to symmetries that were not present
in the general case. Indeed, these symmetries are responsible for the particular 
``geometric flavor'' of the theory -- in particular, for the most important (in our opinion) feature: {\em duality},
and relating all this to exceptional structures (cf.\ \cite{BeKi12}) and Jordan geometries (cf.\ \cite{Be13}).
This is indeed topic for further research; we say some words about this in the last section.

\begin{acknowledgment}
I thank Anders Kock for helpful remarks, and in particular, for pointing out to me his work on pregroupoids.
\end{acknowledgment}

\section{Composition of binary relations  revisited}\label{sec:rels}

\subsection{Binary relations: notation and terminology}
Concerning binary relations,
the following notation and terminology  will be used in all of this work:

\begin{enumerate}
\item
$\Omega,\Omega',\Omega''$  sets considered to be fixed,
\item
$\cP(\Omega)$ the power set (set of subsets) of a set $\Omega$,
\item
$\cR(\Omega,\Omega') := \cP(\Omega \times \Omega')$ set of binary relations between $\Omega$ and $\Omega'$,
\item
{\em graph} of a mapping $f:\Omega \to \Omega'$: 
$\Gr_f := \{ (f(\omega),\omega) \mid \, \omega \in \Omega \} \in \cR(\Omega',\Omega)$
(and every {\em single valued and everywhere defined} relation is of this form),
\item
{\em relational composition} of $b \in \cR(\Omega'',\Omega')$ with  $a \in \cR(\Omega',\Omega)$:
\begin{equation}
\quad b\circ a := ba := \{ (\omega,\alpha) \in \Omega'' \times \Omega \mid \, 
\exists \xi \in \Omega' : \, (\omega , \xi) \in b, (\xi,\alpha) \in a \}
\end{equation}
(so that $\Gr_{g\circ f}=\Gr_g \circ \Gr_f$),
\item
{\em inverse relation} of $a \in \cR(\Omega,\Omega')$:
\begin{equation}
a\inv  := \{ (\alpha,\omega) \mid \, (\omega,\alpha) \in a \} \in \cR(\Omega',\Omega),
\end{equation}
\item
{\em domain} of $a \in \cR(\Omega',\Omega)$:
$\dom (a) = \pr_\Omega (a) = \{ \alpha \in \Omega  \mid \exists \omega \in \Omega'  : (\omega,\alpha)\in a \}$,
\item 
{\em everywhere defined} if $\dom(a) = \Omega$,
\item
{\em image} of $a \in \cR(\Omega',\Omega)$:
$\im (a) = \pr_{\Omega'} (a) = \{ \omega \in \Omega'  \mid \exists \alpha \in \Omega : (\omega,\alpha)\in a \}$,
\item
{\em single valued} if $(\alpha,\omega), (\alpha',\omega) \in a \Rightarrow \alpha = \alpha'$,
\item
{\em intersection} of two relations $a,b \in \cR(\Omega,\Omega')$: 
$a \cap b$ (usual intersection of sets),
\item
{\em natural order} on $\cR(\Omega,\Omega')$: $a \leq b$ if $a \subset b$ (usual inclusion of sets),
\item
$\cR(\Omega):=\cR(\Omega,\Omega)$ set of {\em endorelations} on $\Omega$,
\item
for $x \in \cP(\Omega)$ and $y \in \cP(\Omega')$, we define the {\em all-relation}  by
\begin{equation}
\All_{x,y} := x \times  y \in \cR(\Omega,\Omega'), \qquad
\All_x:= \All_{x,x} = x \times x \in \cR(\Omega), 
\end{equation}
\item
for $x \in \cP(\Omega)$ we define the {\em diagonal of $x$} by 
\begin{equation}
\Dia_x:= \{ (\xi,\xi) \mid \, \xi \in x\} \in \cR(\Omega),
\end{equation}
\item
an endorelation $a \in \cR(\Omega)$ is called:

\ssk
{\em transitive} if $a^2 \subset a$,

  {\em idempotent} if $a^2 = a$,

{\em symmetric} if $a = a\inv$,

{\em image reflexive} if $\Dia_{\im (a)} \subset a$ i.e, $(\omega,\eta) \in a \Rightarrow (\omega,\omega)\in a$,

{\em domain reflexive} if $\Dia_{\dom (a)} \subset a$ i.e, $(\omega,\eta) \in a \Rightarrow (\eta,\eta)\in a$,

{\em reflexive} if $\Dia_\Omega \subset a$, i.e., $(\omega,\omega) \in a$ for all $\omega\in \Omega$,

{\em regular} if $aa \inv a = a$,

{\em equivalence relation on $\Omega$} if it is transitive, symmetric and reflexive,

{\em equivalence relation in $\Omega$} if it is transitive and symmetric. 
\ssk

\item
$\cE^\loc(\Omega)$ set of equivalence relations {\bf in} $\Omega$ (local equivalence relations),
\item
$\cE(\Omega)$ set of equivalence relations {\bf on} $\Omega$,
\item
if $a \in \cE(\Omega)$, we write also $\xi \sim_a \eta$ instead of $(\xi,\eta)\in a$.
The equivalence class of $\omega$ under $a$ is denoted by $[\omega]_a$,
and the canonical projection by
\begin{equation}
\pi_a : \Omega \to \Omega / a,\quad \omega \mapsto [\omega]_a ,
\end{equation}
\item
{\em image of $x \in \cP(\Omega)$ under $a \in \cR(\Omega',\Omega)$}:

$ax := a(x):=  \{ \omega \in \Omega' \mid \exists \xi \in x : (\omega,\xi) \in a \} = \im (a \circ \Dia_x )$,

\item
{\em inverse image} of $x$ under $a$:  $xa:= a\inv x = \dom (\Dia_x \circ a)$.
\end{enumerate}
 
\nin
Note that, if $a$ is an equivalence relation {\bf in} $\Omega$ (so $a=a\inv$, $a^2 \subset a$), then
for all $\xi \in \dom(a)$, we have $(\xi,\xi) \in a$ (proof:  $\exists \eta \in \Omega$:
$(\eta,\xi) \in a$, whence also $(\xi,\eta) \in a$ and $(\xi,\xi) \in a^2 \subset a$),
so $a$ is an equivalence relation {\bf on} $\dom (a)$.

\subsection{(Local) transversality of equivalence relations and sets.}\label{ssec:transv}
An equivalence relation $a$ on $\Omega$ and a subset $x$ of $\Omega$ are called
{\em transversal} if $x$ is a set of representatives of $a$, i.e., $x$ contains exactly one element from
each equivalence class  of $a$. We then write $a \top x$ or $x \top a$.
The set of all systems of representatives of $a$ will be denoted by
\begin{equation}
U_a:= a^\top := \{ x \in \cP(\Omega) \mid \, a \top x \}.
\end{equation}
We say that $a$ and $x$ are {\em locally transversal} (notation $x \top^\loc a$) if $x$ intersects each equivalence class of $a$
at {\em at most one} element: 
$\xi,\xi' \in x, \xi \sim_a \xi' \Rightarrow \xi = \xi'$.
The set of sets that are locally transversal to $a$ will be denoted by
\begin{equation}
U_a^\loc := \{ x \in \cP(\Omega) \mid \, a \top^\loc  x \}.
\end{equation}
If $a,b$ are two equivalence relations on $\Omega$, the set of common systems of representatives, called the
set of {\em bisections of $(a,b)$}, will be denoted by
\begin{equation} \label{eqn:Uab}
U_{ab}:= a^\top \cap b^\top = \{ x \in \cP(\Omega) \mid a \top x, b \top x \} .
\end{equation}
Similarly, the {\em set of local bisections} is defined by
\begin{equation} \label{eqn:Uabloc}
U_{ab}^\loc:= U_a^\loc  \cap U_b^\loc  = \{ x \in \cP(\Omega) \mid a \top^\loc  x, b \top^\loc  x \} .
\end{equation}
To give an equivalent characterisation of (local) transversality,
for any pair $(a,x) \in \cE(\Omega)\times \cP(\Omega)$ define a relation,
called  {\em generalized projection (into $x$ and along $a$)}:
\begin{equation}
P^a_x:= \Dia_x \circ a =  \{ (\xi,\eta)  \mid \, \xi \in x , (\xi,\eta)\in a \} \in \cR(\Omega) .
\end{equation}
The proof of the following lemma is obvious:

\begin{lemma} \label{la:transv1}
A set $x$ is locally transversal to an equivalence relation $a$ if, and only if,
the generalized projection $P^a_x$ is single-valued,
and $a \top x$ iff $P^a_x$ is an operator,  i.e., everywhere defined and single-valued.
\end{lemma}

\nin
Also, $x \top a$ iff $\pi_a \vert_x : x \to \Omega/a$ is bijective, and then
$P^a_x = (\pi_a \vert_x)\inv \circ \pi_a$.

\begin{remark}
It can be shown that, for any pair $(a,x)$, the binary relation $P^a_x$ is
{\em idempotent, regular, and image-reflexive}, and that, conversely, every image-reflexive idempotent regular
binary relation is of this form.
The domain-reflexive idempotent regular relations are then precisely those of the form $(P^a_x)\inv$.
These facts will not be needed in the sequel, but may be useful for a further theory.
\end{remark}

\subsection{Transversality of equivalence relations.}
Two equivalence relations $a,b \in \cE(\Omega)$ will be called {\em transversal} (notation:
$a \top b$) if
each equivalence class of $a$ is a system of representatives of $b$, and vice versa.
This is equivalent to requiring that
\begin{equation}
a \circ b = \All_\Omega \qquad \mbox{ and }  \qquad a \cap b = \Dia_\Omega .
\end{equation}
Note that $\All_\Omega = (\All_\Omega)\inv$, hence the first  condition implies 
$a b = \All_\Omega = ba$, hence $a$ and $b$ {\em commute}.
For instance, if $\Omega$ is the direct product $\Omega_1 \times \Omega_2$ of two sets and
$a,b$ are the equivalence relations given by the fibers of the two projections $\pr_i : \Omega \to \Omega_i$,
then $a,b$ are transversal equivalence relations. 
To see that
this example describes already the general situation,  define, for any
 pair $(a,b) \in \cE(\Omega)^2$, the {\em canonical map} 
\begin{equation}
\pi_{(a,b)} := (\pi_a \times \pi_b) \circ \Delta_\Omega : 
\Omega \to \Omega/a \times \Omega/b, \quad \omega \mapsto ([\omega]_a,[\omega]_b) . 
\end{equation}

\begin{lemma}\label{la:transv2}
For any $(a,b) \in \cE(\Omega)^2$: 
 $a \top b$ if, and only if, $\pi_{(a,b)}$ is bijective.
In this case, an inverse map of $\pi_{(a,b)}$  is given by
 \[
\Omega/a \times \Omega/b \to \Omega, \quad ([\xi]_a ,[\zeta]_b) \mapsto [\xi]_a \cap [\zeta]_b
\]
Summing up, the choice of a transversal pair of equivalence relations is equivalent to the choice of a direct product
structure on the set $\Omega$, and thus $\cP(\Omega)$ is identified with $\cR(\Omega/b,\Omega/a)$.
\end{lemma}

\begin{proof}
The condition $a \circ b = \All_\Omega$ means that $[\xi]_a \cap [\zeta]_b$ is not empty,
for any two $\xi,\zeta \in \Omega$.
The condition $a \cap  b = \Dia_\Omega$ means that $[\xi]_a \cap [\zeta]_b$ contains at most
one element. Thus the map 
$\Omega/a \times \Omega/b \to \Omega$, 
$\omega \mapsto ([\omega]_a,[\omega]_b)$
is well-defined, and it is then clear that it is inverse to $\pi_{(a,b)}$.
\end{proof}

\nin 
Note that every set admits the trivial product structure $(a,b) = (\All_\Omega,\Dia_\Omega)$. 

\begin{definition}\label{def:indrel}
Let $a$ be an equivalence relation in $\Omega$.
Then the following defines an equivalence relation $\bfa$ on $\cP(\Omega)$,
called {\em induced by $a$}:
$$
x \sim_\bfa y \mbox{ iff: } 
a \Dia_x a = a \Dia_y a  \mbox{ iff: } 
\forall \omega \in \Omega: (\exists \xi \in x: \xi \sim_a \omega) \Leftrightarrow (\exists \eta \in y: \eta \sim_a \omega) .
$$
\end{definition}

\nin
Then, in the situation of the preceding lemma, we have:
$\im (x) = \im (y)$ iff $x \sim_\bfb y$, and
$\dom(x)=\dom(y)$ iff $x \sim_\bfa y$.

\subsection{Transversal triples}
A triple $(a,y,c)$ with $a,c \in \cE(\Omega)$ and $y \in \cP(\Omega)$ is called a
{\em transversal triple}
if $a\top c, y \top a, y \top c$.
From the preceding lemma, it follows that $\Omega$ carries a direct product structure, and then $y$
is the graph of a bijection $\Omega /a \to \Omega/b$, whence

\begin{lemma}
A transversal triple on $\Omega$ is the same as a product structure
$\Omega \cong A \times A$.
\end{lemma}

A triple $(a,b,c)$ of equivalence relations is called {\em transversal} if
$a \top b$, $b\top c$, $c \top a$.
This situation is very special: if the triple is non-trivial, it corresponds precisely to the structure 
of a {\em $3$-net} on $\Omega$, and then each equivalence class of each of the three relations carries
the structure of a {\em loop}, see \cite{NS}, p.\ 53/54.

\subsection{Composition revisited}

\begin{theorem}\label{th:torsor0}
Assume that $(a,b)$ is a transversal pair of equivalence relations on $\Omega$, 
and let us identify $\Omega$ with $\Omega/b \times \Omega/a$ via $\pi_{(a,b)}$.
Then $\cP(\Omega)$ is identified with $\cR(\Omega/a,\Omega/b)$.
Therefore, for all   $x,y,z \in \cP(\Omega)$, we may speak of the ternary composition
$xy\inv z \in \cP(\Omega)$. 

\begin{enumerate}
\item
The set $xy\inv z \subset \Omega$  is given by the formula
$$
\boxed{
(xyz)_{ab} :=
\Bigsetof{\omega \in \Omega}
{\begin{array}{c}
\exists \xi \in x,
\exists \eta \in y,
\exists \zeta \in z : \\
\omega \sim_a \zeta, \quad
\eta \sim_a \xi,\quad
\omega \sim_b \xi, \quad
\eta \sim_b \zeta 
\end{array}} }.
$$
It follows that $\cP(\Omega)$ with ternary law $(x,y,z)\mapsto (xyz)_{ab}$ is a semitorsor. 
\item
The set of bisections
$U_{ab}$ defined by (\ref{eqn:Uab}) is stable under the ternary product, and it forms a torsor, isomorphic to the torsor of bijective mappings $\Omega/a \to \Omega/b$ with its usual torsor structure $fg\inv h$. 
In other words, if $(a,y,b)$ is a transversal triple, then $(U_{ab},y)$ is a group, isomorphic to the group of bijections
of $\Omega/a$. 
\item
The set of local bisections $U_{ab}^\loc$ is a pregroupoid when equipped with the induced equivalence relations
$(\bfa,\bfb)$ and the partially defined ternary map
$(\quad)_{ab}$. It is isomorphic to the pregroupoid (``pseudogroup'') of locally defined bijections between $\Omega /a$
and $\Omega/b$.
\end{enumerate}
\end{theorem}


\begin{proof} (1)
According to the lemma, we may assume that
$\Omega=\Omega_1  \times \Omega_2$ and that $a,b$ are given by the two projections:
$\omega \sim_b  \eta$ iff $\omega_1 = \eta_1$, and
$\omega \sim_a \eta$ iff $\omega_2 = \eta_2$.
Applying twice the definition of relational composition, we get 
\begin{align*}
x\circ y^{-1} \circ z &=
\Bigsetof{\omega = (\alpha',\beta') \in \Omega}
{\begin{array}{c}
\exists (\alpha'',\beta'') \in y:
 \\
(\alpha',\beta'') \in x, \quad (\alpha'',\beta') \in z
\end{array}}\,
\\
&=
\Bigsetof{\omega = (\alpha',\beta') \in \Omega}
{\begin{array}{c}
\exists \eta=(\alpha'',\beta'') \in y, \exists \xi \in x, \exists \zeta \in z :
 \\
\xi_1 = \alpha' = \omega_1, \quad 
\xi_2 = \beta'' = \eta_2 ,
\\
\eta_1 = \alpha'' = \zeta_1, \quad 
\zeta_2 = \beta' = \omega_2 
\end{array}}\,
\\
&=
\Bigsetof{\omega  \in \Omega}
{\begin{array}{c}
\exists \eta \in y, \exists \xi \in x, \exists \zeta \in z :
 \\
\xi_1 = \omega_1, \quad 
\xi_2 = \eta_2 , \quad 
\eta_1 = \zeta_1, \quad 
\zeta_2= \omega_2 
\end{array}}\,
\end{align*}
which corresponds to the formula claimed. 
Since, for any two sets $E,F$, the set $\cR(F,E)$ with
$(xyz)=xy\inv x$ is a semitorsor,  $\Omega$ with the ternary product just defined, is, by transport of structure, also
a semitorsor. 

\ssk
(2) Graphs of bijective mappings are precisely the  $x \in \cR(F,E)$ that
are transversal to the equivalence relations given by the two projections; and they form a torsor with respect
to the ternary law $fg\inv h$. (Recall that the empty set is also a torsor; of course, $U_{ab}$ is empty iff $\Omega/a$
and $\Omega/b$ are not equinumerous.) 

\ssk
(3) This follows from Definition \ref{def:indrel} and the remark following the definition. 
\end{proof}

\nin
Recall that the definition of the set $(xyz)_{ab}$ may be described by the parallelogram figure given in the
introduction. 
For the abstract, set-theoretic context such diagrams are quite helpful (see proof of the following theorem); 
however, in the main case of application,
the sets $x,y,z$ are themselves objects having the shape of a ``line'', as in  the following figure:

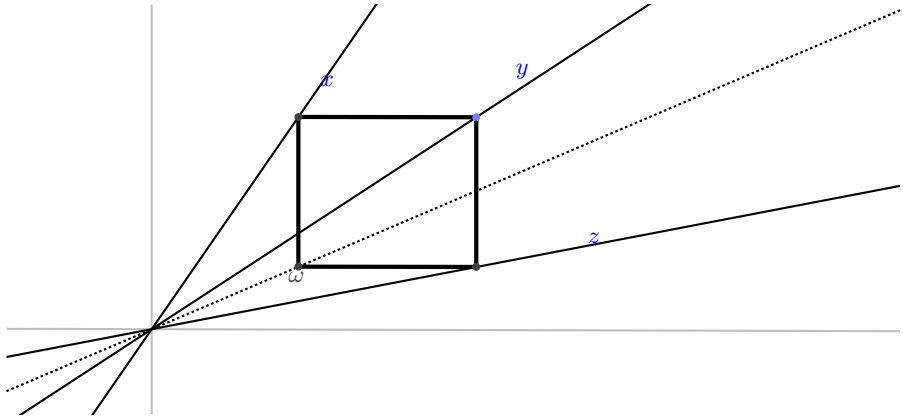
\begin{center}
\newrgbcolor{xdxdff}{0.49 0.49 1}
\psset{xunit=0.5cm,yunit=0.4cm,algebraic=true,dotstyle=o,dotsize=3pt 0,linewidth=0.8pt,arrowsize=3pt 2,arrowinset=0.25}
\begin{pspicture*}(-4.3,-7.32)(19.46,6.3)
\psline[linecolor=lightgray](-0.44,-7.32)(-0.44,6.3)
\psplot[linecolor=lightgray]{-4.3}{19.46}{(-54.2-0.04*x)/12.04}
\psplot{-4.3}{19.46}{(-16.3--7.96*x)/4.4}
\psplot{-4.3}{19.46}{(-42.48--8.38*x)/10.26}
\psplot{-4.3}{19.46}{(-50.63--2.76*x)/11.52}
\psplot[linestyle=dashed,dash=1pt 1pt]{-4.3}{19.46}{(--16.65-2.08*x)/-3.9}
\psline[linewidth=1.6pt](8.19,2.55)(3.46,2.56)
\psline[linewidth=1.6pt](3.46,2.56)(3.46,-2.42)
\psline[linewidth=1.6pt](3.46,-2.42)(8.19,-2.43)
\psline[linewidth=1.6pt](8.19,2.55)(8.19,-2.43)
\begin{scriptsize}
\rput[bl](4.04,3.58){\blue{$x$}}
\rput[bl](9.24,3.82){\blue{$y$}}
\rput[bl](11.16,-1.62){\blue{$z$}}
\psdots[dotstyle=*,linecolor=xdxdff](8.19,2.55)
\psdots[dotstyle=*,linecolor=darkgray](8.19,-2.43)
\psdots[dotstyle=*,linecolor=darkgray](3.46,2.56)
\psdots[dotstyle=*,linecolor=darkgray](3.46,-2.42)
\rput[bl](3.2,-2.92){\darkgray{$\omega$}}
\end{scriptsize}
\end{pspicture*}
\end{center}

\nin
The dotted line is the set $(xyz)$, when $x,y,z$ are secant lines as in the drawing. Indeed, when $x,y,z$ are graphs of linear
functions, with respect to the coordinate axes given by the two grey lines, then
the graph of $xy\inv z$ is constructed exactly is indicated by the scheme.

\subsection{The book-keeping semitorsor structures on $\cP(\Omega)$}
Now we drop
the transversality assumption on $(a,b)$ from the preceding theorem:

\begin{theorem}[The book-keeping semitorsor] \label{th:semitorsor1}
Assume $(a,b)$ is a pair of commuting equivalence relations in $\Omega$, that is, 
$a,b \in \cE^\loc(\Omega)$ and
$ab=ba$. 
For a triple $(x,y,z)$ of subsets of $\Omega$, define a subset
\[
\boxed{
(xyz):=
(xyz)_{ab}^\book  :=
\Bigsetof{\omega \in \Omega}
{\begin{array}{c}
\exists \xi \in x,
\exists \eta \in y, 
\exists \zeta \in z : \\
\omega \sim_a \zeta, \quad
\eta \sim_a \xi,\quad
\omega \sim_b \xi, \quad
\eta \sim_b \zeta 
\end{array}} 
} \, .
\]
Equivalently,  $(xyz)_{ab}^\book$ can be described as direct image of one of the sets $x,y,z$ under a  binary relation,
 by one of the following equivalent formulae
\begin{eqnarray*}
(xyz)_{ab}^\book &= & (b \Dia_x a \cap a \Dia_z b) (y)   \cr
& =& (b \cap a \Dia_z b \Dia_y a) (x)  \cr
& = &(a \cap b \Dia_x a \Dia_y b)(z) .
\end{eqnarray*}
\begin{enumerate}
\item
The law  $(x,y,z) \mapsto (xyz)$ is a para-associative ternary product on $\cP(\Omega)$, and
hence $\cP(\Omega)$ with this law is a semitorsor.
\item
With respect to the induced equivalence relations,
$(\cP(\Omega),\bfa,\bfb,(\quad)_{ab})$  is a
semi-pregroupoid (Definition \ref{def:pregroupoid}).
\item
The whole structure satisfies   the {\em symmetry law}
$$
\boxed{ (xyz)_{ba}^\book = (zyx)_{ab}^\book } .
$$
\end{enumerate}
\end{theorem}

\begin{proof}
Equivalence of the four formulae for $(xyz)_{ab}$ given in the theorem follows directly from the definition of 
relational composition and direct image of sets. 

\ssk
(1) 
We prove by direct computation that $(xy(rst)) = (x(sry)t)$. First, 
\[
(xy(rst)) =
\Bigsetof{\omega \in \Omega}
{\begin{array}{c}
\exists \xi \in x,
\exists \eta \in y,
\exists \zeta \in (rst) : \\
\omega \sim_a \zeta, \quad
\eta \sim_a \xi,\quad
\omega \sim_b \xi, \quad
\eta \sim_b \zeta 
\end{array}} \qquad \qquad \qquad \qquad
\]
\[
= 
\Bigsetof{\omega \in \Omega}
{\begin{array}{c}
\exists \xi \in x,
\exists \eta \in y,
\exists \zeta \in \Omega,
\exists \rho  \in r, \exists \sigma  \in s, \exists \tau \in t  : \\
\omega \sim_a \zeta, \quad
\eta \sim_a \xi,\quad
\omega \sim_b \xi, \quad
\eta \sim_b \zeta,  
\\
\zeta \sim_a \tau,  \quad
\sigma  \sim_a \rho, \quad
\zeta \sim_b \rho , \quad
\tau  \sim_b \sigma  
\end{array}} 
\]
\[
= 
\Bigsetof{\omega \in \Omega}
{\begin{array}{c}
\exists \xi \in x,
\exists \eta \in y,
\exists \zeta \in \Omega,
\exists \rho  \in r, \exists \sigma  \in s, \exists \tau \in t  : 
\\
\omega \sim_a \tau, \quad
\eta \sim_a \xi,\quad
\omega \sim_b \xi, \quad
\eta \sim_b \rho,  
\\
\zeta \sim_a \tau,  \quad
\sigma  \sim_a \rho, \quad
\zeta \sim_b \rho , \quad
\tau  \sim_b \sigma  
\end{array}} 
\]
The condition $\exists \zeta\in \Omega : \zeta \sim_a \tau, \zeta \sim_b \rho$ can be eliminated:
note first that this condition  is equivalent to saying that
$(\tau,\rho) \in ab$.
But according to the condition $\rho \sim_a \sigma$, $\sigma \sim_b \tau$, we  have 
$(\tau,\rho) \in ba$. 
Because of $ba=ab$, the latter condition implies  the former, which thus can be dropped in the description of the set. Thus
\[
(xy(rst))=
\Bigsetof{\omega \in \Omega}
{\begin{array}{c}
\exists \xi \in x,
\exists \eta \in y,
\exists \rho  \in r, \exists \sigma  \in s, \exists \tau \in t  : 
\\
\omega \sim_a \tau, \quad
\eta \sim_a \xi,\quad
\omega \sim_b \xi, \quad
\eta \sim_b \rho,  
\\
\sigma  \sim_a \rho, \quad
\tau  \sim_b \sigma  
\end{array}} 
\]
The preceding computations may be visualized by the following diagrams. In the last diagram, the solid point
in the middle has been eliminated. Note that transitivity and symmetry of $a$ and $b$ have been used several times.

\begin{center}
\psset{xunit=0.5cm,yunit=0.5cm,algebraic=true,dotstyle=o,dotsize=3pt 0,linewidth=0.8pt,arrowsize=3pt 2,arrowinset=0.25}
\begin{pspicture*}(-7.3,-5.32)(19.46,3.3)
\rput[tl](-7.1,0.34){$ (xy(rst))=$}
\psline(-0.52,1.24)(-0.48,-1.5)
\psline(-0.48,-1.5)(1.98,-1.48)
\psline(1.98,-1.48)(1.98,1.22)
\psline(-0.52,1.24)(1.98,1.22)
\pscircle(-0.52,1.24){0.25}
\pscircle(1.98,1.22){0.25}
\pscircle(1.98,-1.48){0.25}
\psline(4.24,1.26)(4.22,-1.46)
\psline(4.22,-1.46)(6.68,-1.44)
\psline(6.68,-1.44)(6.66,1.3)
\psline(4.24,1.26)(6.66,1.3)
\psline(6.68,-1.44)(6.72,-3.84)
\psline(6.72,-3.84)(9.2,-3.82)
\psline(9.2,-3.82)(9.12,-1.44)
\psline(6.68,-1.44)(9.12,-1.44)
\psline(11.08,1.38)(11.1,-1.46)
\psline(11.08,1.38)(13.34,1.38)
\psline(13.34,1.38)(13.4,-3.9)
\psline(11.1,-1.46)(15.84,-1.42)
\psline(15.84,-1.42)(15.84,-3.9)
\psline(13.4,-3.9)(15.84,-3.9)
\pscircle(4.24,1.26){0.25}
\pscircle(6.66,1.3){0.25}
\pscircle(6.72,-3.84){0.25}
\pscircle(9.12,-1.44){0.25}
\pscircle(9.2,-3.82){0.25}
\pscircle(11.08,1.38){0.25}
\pscircle(13.34,1.38){0.25}
\pscircle(15.84,-1.42){0.25}
\pscircle(15.84,-3.9){0.25}
\pscircle(13.4,-3.9){0.25}
\rput[tl](3.08,0.14){$ = $}
\rput[tl](9.94,0.14){$ = $}
\begin{scriptsize}
\psdots[dotstyle=*,linecolor=blue](-0.52,1.24)
\psdots[dotstyle=*,linecolor=blue](-0.48,-1.5)
\rput[bl](-0.84,-2.02){\blue{$\omega$}}
\psdots[dotstyle=*,linecolor=blue](1.98,-1.48)
\psdots[dotstyle=*,linecolor=blue](1.98,1.22)
\rput[bl](-0.44,2.14){$x$}
\rput[bl](1.8,2){$y$}
\rput[bl](1.64,-2.4){$(rst)$}
\psdots[dotstyle=*,linecolor=blue](4.24,1.26)
\psdots[dotstyle=*,linecolor=blue](4.22,-1.46)
\rput[bl](4.1,-2){\blue{$\omega$}}
\psdots[dotstyle=*,linecolor=blue](6.68,-1.44)
\psdots[dotstyle=*,linecolor=blue](6.66,1.3)
\psdots[dotstyle=*,linecolor=blue](6.72,-3.84)
\psdots[dotstyle=*,linecolor=blue](9.2,-3.82)
\psdots[dotstyle=*,linecolor=blue](9.12,-1.44)
\psdots[dotstyle=*,linecolor=blue](11.08,1.38)
\psdots[dotstyle=*,linecolor=blue](11.1,-1.46)
\rput[bl](11.04,-2.02){\blue{$\omega$}}
\psdots[dotstyle=*,linecolor=blue](13.34,1.38)
\psdots[dotstyle=*,linecolor=blue](13.4,-3.9)
\psdots[dotstyle=*,linecolor=blue](15.84,-1.42)
\psdots[dotstyle=*,linecolor=blue](15.84,-3.9)
\rput[bl](4.3,2.14){$x$}
\rput[bl](6.84,2.12){$y$}
\rput[bl](6.36,-4.68){$r$}
\rput[bl](8.62,-0.68){$t$}
\rput[bl](8.94,-4.68){$s$}
\rput[bl](11.04,2.22){$x$}
\rput[bl](13.6,2.18){$y$}
\rput[bl](16.12,-0.58){$t$}
\rput[bl](15.62,-4.82){$s$}
\rput[bl](13.04,-4.74){$r$}
\end{scriptsize}
\end{pspicture*}
\end{center}

\nin Now compute
\[
(x (sry)t) = 
\Bigsetof{\omega \in \Omega}
{\begin{array}{c}
\exists \xi \in x,
\exists \zeta \in (sry),
\exists \tau  \in t : \\
\omega \sim_a \tau, \quad
\zeta \sim_a \xi,\quad
\omega \sim_b \xi, \quad
\tau  \sim_b \zeta  
\end{array}}  \qquad \qquad \qquad\qquad 
\]
\[ =
\Bigsetof{\omega \in \Omega}
{\begin{array}{c}
\exists \xi \in x,
\exists \zeta \in \Omega,
\exists \tau \in t, \exists \sigma \in s,  \eta  \in y,
\exists \rho  \in r : \\
\omega \sim_a \tau, \quad
\zeta  \sim_a \xi,\quad
\omega \sim_b \xi, \quad
\tau  \sim_b \zeta  \\
\eta \sim_a \zeta, \quad \eta \sim_b \rho, \quad \rho \sim_a \sigma ,\quad  \sigma \sim_b \zeta 
\end{array}} 
\]
\[ =
\Bigsetof{\omega \in \Omega}
{\begin{array}{c}
\exists \xi \in x,
\exists \zeta \in \Omega,
\exists \tau \in t, \exists \sigma \in s,  \eta  \in y,
\exists \rho  \in r : \\
\omega \sim_a \tau, \quad
\eta  \sim_a \xi,\quad
\omega \sim_b \xi, \quad
\tau  \sim_b \sigma  \\
\eta \sim_a \zeta, \quad \eta \sim_b \rho, \quad \rho \sim_a \sigma ,\quad  \sigma \sim_b \zeta 
\end{array}} 
\]
Eliminate as above the condition $\exists \zeta \in \Omega: \zeta \sim_a \eta, \eta \sim_b \sigma$ (using $ab=ba$) to get
\[
(x(sry)t)  =
\Bigsetof{\omega \in \Omega}
{\begin{array}{c}
\exists \xi \in x,
\exists \tau \in t, \exists \sigma \in s,  \eta  \in y,
\exists \rho  \in r : \\
\omega \sim_a \tau, \quad
\eta  \sim_a \xi,\quad
\omega \sim_b \xi, \quad
\tau  \sim_b \sigma  \\
\eta \sim_b \rho, \quad \rho \sim_a \sigma 
\end{array}} 
\]
agreeing with the preceding. 
The following diagram visualizes the computation.

\begin{center}
\psset{xunit=0.5cm,yunit=0.5cm,algebraic=true,dotstyle=o,dotsize=3pt 0,linewidth=0.8pt,arrowsize=3pt 2,arrowinset=0.25}
\begin{pspicture*}(-6.56,-3.5)(17.97,3.05)
\rput[tl](-5.79,-1.12){$ (x(sry)t) = $}
\psline(0.36,-2.48)(0.32,-0.28)
\psline(0.32,-0.28)(2.32,-0.26)
\psline(2.32,-0.26)(2.36,-2.44)
\psline(0.36,-2.48)(2.36,-2.44)
\psline(4.72,-0.28)(4.74,-2.34)
\psline(4.74,-2.34)(6.64,-2.32)
\psline(6.6,-0.24)(6.64,-2.32)
\psline(4.72,-0.28)(6.6,-0.24)
\psline(6.6,-0.24)(6.56,1.8)
\psline(6.56,1.8)(8.52,1.78)
\psline(8.52,1.78)(8.56,-0.22)
\psline(6.6,-0.24)(8.56,-0.22)
\psline(11.67,-2.36)(11.63,-0.29)
\psline(11.67,-2.36)(13.51,-2.36)
\psline(13.51,-2.36)(13.43,1.75)
\psline(13.43,1.75)(15.49,1.73)
\psline(15.49,1.73)(15.53,-0.29)
\psline(15.53,-0.29)(11.63,-0.29)
\pscircle(0.32,-0.28){0.25}
\pscircle(2.32,-0.26){0.25}
\pscircle(2.36,-2.44){0.25}
\pscircle(4.72,-0.28){0.25}
\pscircle(6.64,-2.32){0.25}
\pscircle(6.56,1.8){0.25}
\pscircle(8.52,1.78){0.25}
\pscircle(8.56,-0.22){0.25}
\pscircle(11.63,-0.29){0.25}
\pscircle(13.51,-2.36){0.25}
\pscircle(13.43,1.75){0.25}
\pscircle(15.49,1.73){0.25}
\pscircle(15.53,-0.29){0.25}
\rput[tl](3.53,-1.1){$ = $}
\rput[tl](9.95,-1.06){$ = $}
\begin{scriptsize}
\psdots[dotstyle=*,linecolor=blue](0.36,-2.48)
\rput[bl](0.11,-2.94){\blue{$\omega$}}
\psdots[dotstyle=*,linecolor=blue](0.32,-0.28)
\psdots[dotstyle=*,linecolor=blue](2.32,-0.26)
\psdots[dotstyle=*,linecolor=blue](2.36,-2.44)
\psdots[dotstyle=*,linecolor=blue](4.72,-0.28)
\psdots[dotstyle=*,linecolor=blue](4.74,-2.34)
\psdots[dotstyle=*,linecolor=blue](6.64,-2.32)
\psdots[dotstyle=*,linecolor=blue](6.6,-0.24)
\psdots[dotstyle=*,linecolor=blue](6.56,1.8)
\psdots[dotstyle=*,linecolor=blue](8.52,1.78)
\psdots[dotstyle=*,linecolor=blue](8.56,-0.22)
\psdots[dotstyle=*,linecolor=blue](11.67,-2.36)
\psdots[dotstyle=*,linecolor=blue](11.63,-0.29)
\psdots[dotstyle=*,linecolor=blue](13.51,-2.36)
\psdots[dotstyle=*,linecolor=blue](13.43,1.75)
\psdots[dotstyle=*,linecolor=blue](15.49,1.73)
\psdots[dotstyle=*,linecolor=blue](15.53,-0.29)
\rput[bl](-0.15,0.41){$x$}
\rput[bl](1.99,0.53){$(sry)$}
\rput[bl](2.05,-3.35){$t$}
\rput[bl](4.24,0.49){$x$}
\rput[bl](6.36,-3.24){$t$}
\rput[bl](6.2,2.55){$s$}
\rput[bl](8.26,2.53){$r$}
\rput[bl](8.28,-1){$y$}
\rput[bl](11.37,0.49){$x$}
\rput[bl](13.43,-3.24){$t$}
\rput[bl](13.25,2.63){$s$}
\rput[bl](15.35,2.61){$r$}
\rput[bl](15.31,-1.1){$y$}
\end{scriptsize}
\end{pspicture*}
\end{center}

\nin   
The last diagram is obviously equivalent to the final diagram in the preceding figure. 
Similarly, it is seen that
$((xyr)st)= (x(sry)t)$.
This proves para-associativity.

\ssk (2)
The compatibility of the equivalence relations $\bfa$ and $\bfb$ with the ternary law, as
required by item (C) in Definition \ref{def:pregroupoid}, follows directly from the definitions, and
(3)  follows directly from the formula defining $(xyz)_{ab}^\book$. 
\end{proof}

\nin
We call the semi-torsor and semi-pregroupoid structure defined above the
{\em book-keeping semi-torsor (semi-pregroupoid)} structure on $(\Omega,a,b)$,
because it  underlies all of the  more structured ternary products to be defined in the sequel, and
it formalizes the necessary book-keeping for these  structures to be well-defined.

\begin{example}
If $a=b$, writing $\pi:\Omega \to \Omega/a$, one gets from the theorem
$$
(xyz)_{aa} =
\pi\inv \bigl( \pi(x) \cap \pi(y) \cap \pi(z) \bigr) = a(x) \cap a(y) \cap a(z),
$$
which is indeed a commutative and para-associative operation on $\cP(\Omega)$. 
Thus set-theoretic intersection appears as a sort of contraction of composition of relations.
However, if  $x,y,z$ are transversal to $a$, then $a(x)=\Omega$, etc., so
  $(xyz)_a = \Omega$, so the ternary product is rather degenerate. 
\end{example}

\begin{example}
If $a =\Dia_\Omega$ and $b$ any equivalence relation,
one gets
$$
(xyz)_{ab} = \pi\inv (\pi (x \cap y \cap z)) = b (x \cap y \cap z)  ,
$$
again a commutative and para-associative operation on $\cP(\Omega)$.
More generally, if $a \subset b$, the operation $(xyz)_{ab}$ can be described by multiple intersections of sets.
Similar remarks can be made if $a =\All_\Omega$.
\end{example}

\begin{remark}
More generally,  binary, or general $n$-ary, composition may be described in a similar way
as an $n$-ary operation on $\cP(\Omega)$
 depending on the choice of several (commuting) equivalence relations, in such a way that the resulting
product remains associative even if the equivalence relations are not transversal; if all equivalence relations
are the same, or some are reduced to the identity relation,
 the associative operation reduces to multiple intersections, in a similar way  as in the examples described above.
 However, the ternary composition $xy\inv z$ seems to be the most interesting among these operations. 
\end{remark}


\section{Torsors and semi-torsors defined by  (semi-)pregroupoids}
\label{sec:Bisections}

In this section we assume   that
$(\Omega,a,b,[\quad])$ is a semi-pregroupoid.  We will show that
 this structure carries  over to $(\cP(\Omega),\bfa,\bfb)$;
and if $(\Omega,a,b,[\quad])$ is a pregroupoid, then this carries over
to $(U_{ab}^\loc,\bfa,\bfb)$:

\begin{theorem}[Associative geometry of a (semi-)pregroupoid]\label{th:semitorsor2}\label{th:torsor1}
Assume $(\Omega,a,b,[\quad])$ is a semi-pregroupoid.
For $(x,y,z)  \in \cP(\Omega)^3$ let
\[
\boxed{
(xyz)_{ab}:=
\Gamma(x,a,y,b,z) :=
\Bigsetof{\omega \in \Omega}
{\begin{array}{c}
\exists \xi \in x,
\exists \eta \in y,
\exists \zeta \in z : \\
\eta \sim_a \xi,\quad
\eta \sim_b \zeta ,
\quad \omega = [\xi \eta \zeta]
\end{array}}  } .
\]
\begin{enumerate}
\item
The law $(x,y,z)\mapsto (xyz)_{ab}$ defines a semitorsor structure on $\cP(\Omega)$.
\item
With respect to the induced equivalence relations, 
$(\cP(\Omega),\bfa,\bfb,(\quad)_{ab})$ is  a
semi-pregroupoid.  
\item
The set 
$U_{ab}^\loc$ of local bisections of $(a,b)$ is stable under the ternary product $(xyz)_{ab}$, and becomes 
a semitorsor with respect to this product. When equipped with the equivalence relations $(\bfa,\bfb)$, it
becomes a semi-pregroupoid. 
\item
The set $U_{ab}$ of bisections of $(a,b)$ is stable under the ternary product $(xyz)_{ab}$ and becomes a
semitorsor with respect to this product.
\item
If $a \top b$,  the  law $(\quad)_{ab}$ coincides with the one considered in Theorem \ref{th:torsor0}.
\end{enumerate}
Assume, moreover, that $(\Omega,a,b,[\quad])$ is a pregroupoid.
Then the semi-pregroupoid  $U_{ab}^\loc$ defined above is also a pregroupoid, and the semitorsor
$U_{ab}$ is  a torsor.
\end{theorem}

\begin{proof} (1)
We prove the para-associative law
$(xy(rst))=(x(sry)t)$. 
Note first that, by the defining property (C) of a semi-pregroupoid, with notation from the theorem, the missing two
book-keeping conditions $\omega \sim_a \zeta$ and $\omega \sim_b \xi$ also hold, so that we have in fact
\begin{equation}\label{eqn:1}
\boxed{
(xyz)_{ab}=
\Bigsetof{\omega \in \Omega}
{\begin{array}{c}
\exists \xi \in x,
\exists \eta \in y,
\exists \zeta \in z : \\
\eta \sim_a \xi,\quad
\eta \sim_b \zeta , \quad
\xi \sim_b \omega, \quad
\omega \sim_a \zeta 
\\ \omega = [\xi \eta \zeta] 
\end{array}}    } \, .
\end{equation}
Therefore
the first  computation from the proof of Theorem \ref{th:semitorsor1} shows that
\[
(xy(rst))=
\Bigsetof{\omega \in \Omega}
{\begin{array}{c}
\exists \xi \in x,
\exists \eta \in y,
\exists \rho  \in r, \exists \sigma  \in s, \exists \tau \in t  : 
\\
\omega \sim_a \tau, \quad
\eta \sim_a \xi,\quad
\omega \sim_b \xi, \quad
\eta \sim_b \rho,  
\quad
\sigma  \sim_a \rho, \quad
\tau  \sim_b \sigma  \\
\omega = [\xi \eta [\rho \sigma  \tau]]
\end{array}} 
\]
On the other hand, the second computation  from the proof of Theorem \ref{th:semitorsor1} yields
\[
(x(sry)t)  =
\Bigsetof{\omega \in \Omega}
{\begin{array}{c}
\exists \xi \in x,
\exists \tau \in t, \exists \sigma \in s,  \eta  \in y,
\exists \rho  \in r : \\
\omega \sim_a \tau, \quad
\eta  \sim_a \xi,\quad
\omega \sim_b \xi, \quad
\tau  \sim_b \sigma  \quad 
\eta \sim_b \rho, \quad \rho \sim_a \sigma 
\\
\omega = [\xi [ \sigma \rho  \eta ]  \tau]
\end{array}} 
\]
By the para-associative law for $[\quad]$, 
we have equality of both sets (and it becomes obvious from the proof why we call the first four conditions 
``book-keeping conditions'').
(Note: the condition $ab=ba$ is not stated as assumption in the theorem since it is a consequence of the other assumptions,
see Lemma \ref{la:abcommute}.)
 
 \ssk
(2)
Assume $x \sim_\bfa y$ and $y\sim_\bfb  z$ and let us show that $(xyz) \sim_\bfa  z$, that is,
every element of $(xyz)$ is $a$-related to an element of $z$, and conversely.
The first statement follows directly from the fact that $\omega = [\xi \eta \zeta] \sim_a \zeta$ (defining property of a
semi-pregroupoid). In order to prove the converse,
let $\zeta \in z$. By assumption, there are $\eta \in y$ and $\xi \in x$ such that 
$\eta \sim_b \zeta$ and $\xi \sim_a \eta$,
thus $\omega := [\xi \eta \zeta]$ is an element of $(xyz)$, and 
$\omega \sim_a \zeta$, as seen above. 
In the same way it is seen that $(xyz) \sim_\bfb  x$.

\ssk
(3)
Assume $x,y,z$ are local bisections, and let us show that $(xyz)_{ab}$ is again a local bisection.
Let $\omega = [\xi \eta \zeta]$ and $\omega' = [\xi' \eta' \zeta'] \in (xyz)_{ab}$ such that
$\omega \sim_a \omega'$ 
 (with $\xi,\xi'$, etc, satisfying the
conditions from (\ref{eqn:1}).
Among these conditions are $\omega \sim_a \xi$, $\omega' \sim_a \xi'$; 
from transitivity of $a$, we get $\xi \sim_a \xi'$, whence $\xi=\xi'$ since $x$ is a local bisection.
This implies $\eta \sim_b \xi = \xi' \sim_b \eta'$,
whence $\eta = \eta'$ since $y$ is a local bisection, and in the same way, $\zeta = \zeta'$, whence
$\omega = \omega'$, showing that  $(xyz)_{ab}$ is a local section of $a$.
In the same way we see that it is a local section of $b$, whence $(xyz)_{ab} \in U_{ab}^\loc$.
Thus $U_{ab}^\loc$ is stable under the product and hence forms a sub-semitorsor of $\cP(\Omega)$
and a sub-semi-pregroupoid of $(\cP(\Omega),\alpha,\beta)$. 

\ssk
(4)
Since $U_{ab} = \{ x \in U_{ab}^\loc \mid x \sim_\alpha \Omega \}$, the statement follows from (3).

\ssk
(5) 
Assume that $a \top b$. 
Then $(\Omega ,a,b,[\quad])$ is isomorphic to a pair pregroupoid, see 
 example \ref{ex:pairpregroupoid}, that is,  the equivalence relations $(a,b)$ determine the algebraic ternary law
$[xyz]$ entirely; therefore we are back in the purely set-theoretic setting of the preceding chapter.

\ssk
Now assume $\Omega$ is a pregroupoid. All that remains is to prove is the idempotent law for
$U_{ab}^\loc$.
Let us show that
$(xxz)_{ab} = z$ if $x \sim_\bfb z$.
Indeed,
if $\omega = [\xi \eta \zeta]$ with $\xi, \eta \in x$,
then from $\xi \sim_\eta \eta$ we get $\xi =\eta$ (since $x$ is a local section), whence
$\omega = [\xi \xi \zeta]=\zeta \in z$, thus
$(xxz)_{ab} \subset z$.
To prove the other inclusion, we have to use that, by assumption,  $x \sim_\bfb z$:
for every $\zeta \in z$, there is $\xi \in x$ with $\xi \sim_\beta z$, whence
$\zeta = [\xi \xi \zeta] \in (xyz)_{ab}$, whence $z \subset (xxz)_{ab}$.
Similarly, if $x \sim_\bfa z$, we get $(xzz)_{ab}=x$.
This proves the idempotent law for $U_{ab}^\loc$ and at the same time for $U_{ab}$.
\end{proof}

\begin{example}\label{ex:singletons}
Singletons (sets with one element) are special local bisections, and restricted to the set of singletons, 
the induced relations $(\bfa,\bfb)$ are just $(a,b)$, and the law $(\quad)_{ab}$ reduces to $[\quad]$.
In this sense, the original pregroupoid $\Omega$ is a sub-pregroupoid of $U_{ab}^\loc$.
Its union  with  $\emptyset$ is a semi-torsor (this generalizes the observation, groing back to Baer, that a groupoid can be 
completed by $\emptyset$ into a semigroup, 
see \cite{CW} and  \cite{Ba29}).
\end{example}

The sets $U_a,U_b$ of sections in a pregroupoid $(\Omega,a,b,[\quad])$ do in general not carry a group structure.
However, in the following cases there is such a structure:

\begin{corollary}\label{cor:groupUa}
Assume that $(\Omega,a,[\quad])$ is an assocoid that is either
\begin{enumerate}
\item
a  {\em torsor-bundle}, i.e., a pregroupoid with $a=b$, or
\item
a {\em (left) principal equivalence relation (prev)}, i.e. a pregroupoid with $(a,b)=(a,\Omega)$
(cf.\ Appendix \ref{App:B}), or
\item
a {\em (right) prev}, i.e., a pregroupoid with $(a,b)=(\Omega,a)$.
\end{enumerate}
In all three cases, the space of sections of $a$ carries a natural torsor structure, which can equivalently be
described as the space of sections  $s: B \to \Omega$ of the canonical projection $\pi:\Omega \to B$  with ``pointwise product'' 
\begin{equation}\label{eqn:section}
(rst) (b):= [r(b) \, s(b) \,t(b)] ,
\end{equation}
or as the set $U_a$ with ternary product  given by
\[
(xyz)_{a}=
\Bigsetof{\omega \in \Omega}
{\begin{array}{c}
\exists \xi \in x,
\exists \eta \in y,
\exists \zeta \in z : \\
\eta \sim_a \xi,\quad
\eta \sim_a \zeta ,
\quad \omega = [\xi \eta \zeta]_a 
\end{array}}  .
\]
Moreover, the last formula defines a semitorsor structure on all of $\cP(\Omega)$, and a 
pregroupoid structure on $(\cP(\Omega),\bfa)$, and  in cases (2) and (3),
$(U_a^\loc, \bfa, (\quad)_a)$ is then  a again a prev.
\end{corollary}

\begin{proof}
A pregroupoid with $a=b$ is the same as a {\em torsorbundle over $\Omega/a$},
and then a bisection is the same as a section of $a$, i.e., $U_a = U_aa$.
The equivalence of both descriptions is immediate.
Thus, in this case, the result follows immeditely from the preceding theorem (and it also clear from the first 
description).

If $b=\Omega$, then $(\Omega,a,[\quad])$ is a {\em (left) principal equivalence relation (prev)}.
By restriction of the domain of definition of $[\quad]$, it is then also a torsorbundle, and hence in this case
also the space of sections $U_a = U_{aa}$ carries a torsor structure (but then the product $(\quad)_{aa}$ is not to be
confused with the product $(\quad)_{a\Omega}$, which is uninteresting) . 
The final statement also follows  directly from the preceding theorem.
\end{proof}

\begin{remark}
The semitorsor law on $\cP(\Omega)$ defined in the corollary has the following interpretation: 
it is the ``pointwise semitorsor loi of sets'' -- 
recall first that, for any group $(G,+)$, the power set $\cP(G)$ becomes a semigroup with
law $A + B = \{ a + b \mid a \in A,b \in B\}$; likewise, for any torsor $G$, the power set $\cP(G)$
becomes a semitorsor with ``pointwise product''.
Now, in the setting of the corollary, fix some equivalence class $[o]$ of $a$ and let
$x_o := x \cap [o]$, etc.
Then the formula given in the theorem says that
$(xyz)_a \cap [o]$ is empty if one of the sets $x_o,y_o,z_o$ is empty,
and given by the ``pointwise set product''
$[x_o y_o; z_o]$ in the torsor $[o]$ else.
\end{remark}

\begin{example}[Case of a group.]
See section \ref{sec:groupcase}.
Note that, in this case, a subtelety shows up:
one subgroup $A$ of $\Omega$ defines two equivalence relations $a$ (right relation) and $b=:\hat a$ (left) relation, which
commute.
Thus $(\Omega,a,\hat a,[\quad])$ is a pregroupoid, giving rise to a torsor
$U_{a\hat a}$, which corresponds to the ``balanced torsor $U_{aa}$'' from \cite{Be12}.
On the other hand, since $a$ and $\hat a$ are prev's in their own right, we may also define the torsors
$U_a$ and $U_{\hat a}$, corresponding to the {\em unbalanced torsors} from \cite{Be12}.
Thus one subgroup defines three different torsors (they coincide if $\Omega$ is abelian). 
\end{example}

\section{Structure of $(U_a,U_b)$: canonical kernel,  associative pair}
\label{sec:pair}

We continue to assume that $(\Omega,a,b,[\quad])$ is a (semi-)pregroupoid.
In the special case of an {\em associative geometry} in the sense of  \cite{BeKi10}, there is an associated
{\em tangent object} of the geometry, comparable to the
Lie algebra of a Lie group: this tangent object is an {\em associative pair}.
In the present, much more general, context, the pair of sets $(U_a,U_b)$ plays a similar r\^ole. 
Understanding the structure of $(U_a,U_b)$ comprises the study of the torsor
$U_{ab} = U_a \cap U_b$. Again, in the framework of \cite{BeKi10}, the structure of the group 
$(U_{ab},y)$ is quite
well understood using a morphism $U_{ab} \to \Bij(y)$, called the {\em canonical kernel}.
Part of this generalizes to the present  context.

\subsection{Generalized associative pairs}

\begin{definition}
A {\em (generalized) associative pair} is a pair of sets $(U^+,U^-)$ together with two ternary maps
$$
U^\pm \times U^\mp \times U^\pm \to U^\pm, \quad
(x,y,z) \mapsto \langle xyz \rangle^\pm
$$
satisfying the para-associative law 
$$
\langle xy \langle uvw \rangle^\pm \rangle^\pm =
\langle x \langle vuy \rangle^\mp w \rangle^\pm =
\langle \langle xyu \rangle^\pm vw \rangle^\pm.
$$
\end{definition}

The adjective  ``generalized'' will
 only be used in contexts where there could be confusion with ``usual'' asociative pairs 
 (as defined in \cite{Lo75}; cf.\ \cite{BeKi10}), namely when
 $U^\pm$ are modules over a ring $\bK$, where the
 ``usual'' definition requires the ternary product  to be $\bK$-trilinear.
 
\begin{example}
Every torsor $G$ with $U^+=G=U^-$ and
$\langle \quad \rangle^+ = (\quad ) =\langle \quad \rangle^-$ is an associative pair.
\end{example}

\begin{example}\label{ex:ap}
Let $E,F$ be sets and $(U^+,U^-) = (\Map(E,F),\Map(F,E))$  and
\begin{align*}
\Map( E, F) \times \Map(F,E) \times \Map(E,F) \to \Map(E,F), & \quad (f,g,h) \mapsto fgh, 
\cr
\Map(F,E) \times \Map(E,F) \times \Map(F,E) \to \Map(F,E), & \quad (f,g,h) \mapsto hgf .
\end{align*}
This is an associative pair.
If $E,F$ happen to be $\bK$-modules, the ternay products are  linear in only one of the arguments!
\end{example}

\begin{theorem}
For any pregroupoid $(\Omega,a,b,[\quad])$,   the pair 
$(U^+,U^-): = (U_a,U_b)$ is an associative pair with product given by $(\quad)_{ab}$: the maps 
$$
U_a \times U_b  \times U_a  \to U_a, \quad
(x,y,z) \mapsto \langle xyz\rangle^+ := (xyz)_{ab}
$$
$$
U_b   \times U_a  \times U_b   \to  U_b, \quad
(x,y,z) \mapsto \langle xyz \rangle^- := (xyz)_{ab}
$$
are well-defined and para-associative.
\end{theorem}

\begin{proof}
Note first that this statement is stronger than the result above saying that $U_a \cap U_b$ is stable under the ternary
product.
We will give a proof by using an operator calculus, which may have some interest in its own right.
Let us define, for any $x,z \in \cR(\Omega)$, the following binary relation 
\begin{align*}
M_{xabz} &:=
\Bigsetof{ (\omega,\eta) \in \Omega^2}
{\begin{array}{c}
\exists \xi \in x,
\exists \zeta \in z : \\
\eta \sim_a \xi,\quad
\eta \sim_b \zeta ,
\quad \omega =  [\xi  \eta \zeta]
\end{array}}  
\cr
&
=
\Bigsetof{ (\omega,\eta) \in \Omega^2}
{\begin{array}{c}
\exists \xi ,  \zeta \in \Omega  : \\
(\xi,\eta)  \in P^a_x,\quad
(\zeta,\eta) \in P^b_z,
\quad \omega = [\xi \eta  \zeta]
\end{array}}  .
\end{align*}
Then it follows directly from the definitions  that, for any $y \in \cP(\Omega)$, 
\begin{equation}
(xyz)_{ab} = M_{xabz}(y).
\end{equation}
Now, if $a \top x$ and $b \top z$,
then $\xi = P^a_x(\eta)$ and $\zeta = P^b_z(\eta)$ exist  and are uniquely determined, for all $\eta \in \Omega$,
and hence 
$$
\omega = [P^a_x(\eta) \, \eta \,  P^b_z(\eta)]
$$
also exists and is uniquely determined, which means that $M_{xabz}$ is a mapping.
Similar arguments apply for left- and right multiplications. Let us  summarize:

\begin{lemma}\label{la:MLR}
Define operators on $\Omega$ as follows:
\begin{enumerate}
\item
if $x \top a$, $z \top b$: 
$\qquad M_{xabz} : \Omega \to \Omega$, $\quad \eta \mapsto [P^a_x(\eta)\,  \eta \,  P^b_z(\eta)]$ 
\item
if $x\top a$, $y \top b$:
$\qquad L_{xayb}: \Omega \to \Omega$, $\quad \zeta \mapsto [P^a_x   P^b_y(\zeta) \, P^b_y(\zeta) \, \zeta]  $
\item
if $y \top a$, $z \top b$:
$\qquad R_{zbya}:\Omega \to \Omega$, $\quad \xi \mapsto [\xi \,  P^a_y(\xi) \, P^b_z P^a_y (\xi)]$
\end{enumerate}
Then we have, under the respective transversality conditions,
$$
(xyz)_{ab} = L_{xayb}(z) , \quad
(xyz)_{ab} =  M_{xayz}(y), \quad
(xyz)_{ab}= R_{zbya}(x),
$$
and the operators satisfy,  whenever the suitable  transversality conditions hold, the ``semitorsor relations'': 
\[
R_{aubz} L_{xavb}= M_{xabz}M_{uabv} = L_{xavb}R_{aubz} ,
\quad L_{xayb} L_{uavb} = L_{L_{xayb}(u)avb}  . 
\]
\end{lemma}

\begin{proof}
The ``semitorsor  relations'' are simply the translation of the para-associative law of the semitorsor defined by
$(a,b)$ in the language of left-, right- and middle translations (justified by the fact that, applied to singletons, these
operators must satisfy para-associativity, by Theorem \ref{th:semitorsor2}).
\end{proof}

With some care, the operator $M_{xabz}$ may be written in argument-free notation:
$$
M_{xabz} = [ P^a_x, \id_\Omega ,P^b_z]
$$
which is close to the formula $M_{xabz} = P^a_x - \id + P^b_z$ from \cite{BeKi10}.

\begin{lemma}\label{la:Idempot}
For all $x \in U_{ab}$, we have $L_{xaxb} = \id_\Omega$ and $R_{xbxa}=\id_\Omega$.
\end{lemma}

\begin{proof}
Directly from the definition of the projection operators one sees that, if $x \in U_{ab}$, then
$P^a_x \circ P^b_x = P^b_x$, whence
$$
L_{xaxb} (\zeta) = [P^a_x P^b_x (\zeta) \,  P^b_x(\zeta) \,  \zeta ]   = [P^b_x(\zeta) \, P^b_x(\zeta)\, \zeta ] = \zeta
$$
and similarly for right multiplication operators.
\end{proof}

\begin{lemma}
For all $x,y \in U_{ab}$, the operator 
$L_{xayb}$ is invertible with inverse $L_{yaxb}$, and $R_{xbya}$ is invertible with inverse $R_{ybxa}$.
\end{lemma}

\begin{proof}
The semitorsor relation from Lemma \ref{la:MLR} and the preceding lemma  imply  
$$
L_{xayb} L_{yaxb} = L_{(xyy)_{ab} a x b} = L_{xaxb} = \id_\Omega ,
$$
whence the claim for the $L$-operator. Similarly for the $R$-operator.
\end{proof}

\begin{lemma}\label{la:La}
\begin{enumerate}
\item
If $x \top a$, $y \top b$, then $L_{xayb} \in \Aut_1(a)$, i.e., it preserves all equivalence classes of $a$,
cf.\ Def.\ \ref{def:Auta}.
\item
If $y \top a$, $z \top b$, then $R_{zbya} \in \Aut_1(b)$.
\end{enumerate}
\end{lemma}

\begin{proof} (1)
For any $\zeta \in \Omega$, the element
$\omega = L_{xayb}(\zeta) = [P^a_x P^b_y(\zeta) \, P^b_y(\zeta)\, \zeta]$
belongs to the same $a$-equivalence class as $\zeta$, by the defining  property (C)  of a pregroupoid.
Similarly for (2).
\end{proof}

\nin
Now we prove the theorem:  let $x, z \in U_{a}$ and $y \in U_b$. 
Since $g:=L_{xayb}$ is a bijection of $\Omega$ preserving each equivalence class of $a$,
the image $g(z)$ contains again exactly one element from each equivalence class of $a$, i.e.,
$g(z) \top a$, that is, $(xyz)_{ab} \top a$.
Thus the first of the two maps is well-defined, and so is  the second. 
\end{proof}

\begin{example}
Assume $a \top b$.  Then with  $E:=\Omega/b$, $F:=\Omega/a$
we get the associative pair from Example \ref{ex:ap}.
\end{example}

\begin{remark}
On could define ``associative pair-oids'' and then state and prove an ``oid''-version of the preceding result,
but we will not spell out this here.
\end{remark}

\subsection{Automorphisms, and self-distributivity}
A {\em morphism of a partially defined ternary product map} is a map preserving domains of definition and
commuting with the algebraic product maps, and the {\em automorphism group} of a 
(semi-) pregroupoid is then denoted by $\Aut(\Omega) = \Aut(\Omega,D,[\quad])$.
Note that an automorphism $g \in \Aut(\Omega)$ need not preserve the equivalence relations $a,b$
individually, but only the set $D=\Omega \times_a \Omega \times_b \Omega$.
Thus we end up with several  kinds of  automorphism groups:
besides $\Aut(\Omega)$, there are also
\begin{eqnarray*}
\Aut(\Omega,a) = \Aut(\Omega)\cap \Aut(a) , & &
\Aut(\Omega,b) = \Aut(\Omega)\cap \Aut(b) ,
\cr
\Aut_1(\Omega,a) =  \Aut(\Omega) \cap \Aut_1(a) ,
& &
\Aut_1(\Omega,b) =  \Aut(\Omega) \cap \Aut_1(b).
\end{eqnarray*}

\begin{lemma}\label{la:La2}
Under the assumptions of Lemma \ref{la:La}:
 $L_{xayb}$ permutes equivalence classes of $b$, that is,
$L_{xayb} \in \Aut(b)$, and
$R_{zbya}$ permutes equivalence classes of $a$, that  is,
$R_{zbya} \in \Aut(a)$.
\end{lemma}

\begin{proof}
Assume $\zeta \sim_b \zeta'$.
It follows that $P^b_y(\zeta)=P^b_y(\zeta')$, and therefore
$$
L_{xayb}(\zeta) = [P^a_x P^b_y(\zeta) \, P^b_y(\zeta)\, \zeta] \sim_b P^a_x P^b_y(\zeta) =P^a_x P^b_y(\zeta')
$$
$$ \qquad \qquad \qquad \qquad 
  \sim_b
[P^a_x P^b_y(\zeta') \, P^b_y(\zeta')\, \zeta'] =
L_{xayb}(\zeta')
$$
whence $L_{xayb} \in \Aut(b)$.
Similarly, $R_{zbya} \in \Aut(a)$. 
\end{proof}

\begin{theorem}\label{th:self-distrib}
The left- and right multiplication operators $L_{xayb}$ and $R_{zbya}$ defined in lemma \ref{la:MLR}
(under the transversality conditions given there) are automorphisms of $(\quad)_{ab}$,
which corresponds to the ``self-distributivity'' identities
$$
(xy(uvw)_{ab})_{ab}   =  ((xyu)_{ab} (xyv)_{ab} (xyw)_{ab})_{ab} ,
$$
$$
((uvw)_{ab} yz)_{ab} = ((uyz)_{ab} (vyz)_{ab} (wyz)_{ab} )_{ab} . 
$$
It follows that left- and right multiplicaiton operators belong to the group
$\Aut(\Omega,a) \cap \Aut(\Omega,b)$.
\end{theorem}

\begin{proof}
The self-distributivity identities hold in any associoid, cf.\ equation
(\ref{eq:autos}). Next,
recall (example \ref{ex:singletons})
 that the union of the set of all singletons with $\emptyset$ forms a sub-semitorsor of the semitorsor of local
bisections, and thus the self-distributive identies hold with respect to singletons $u,v,w$, which means that
left and right translations act by automorphisms of $[\quad]$, i.e., they belong to $\Aut(\Omega,[\quad])$.
As seen in the preceding lemma, they belong also to  $\Aut(a) \cap \Aut(b)$.
\end{proof} 

\begin{lemma}
Let $g \in \Aut(\Omega,a) \cap \Aut(\Omega,b)$. Then $g$ is an automorphism of the ternary product
$(\quad)_{ab}$: $ g (xyz)_{ab} = (g(x) \, g(y) \, g(z) )_{ab})_{ab}$. 
\end{lemma}

\begin{proof}
The proof is by a straightforward change of variables, $\xi' = g\inv (\xi)$, etc.,
using that $\xi \sim_a \eta$ iff $\xi' \sim_a \eta'$ (since $g$ permutes equivalence classes of $a$).
\end{proof}

\begin{remark}
In a similar way, the middle multiplication operators $M_{xabz}$ preserve the domain of definition $D$ and exchange $x$ and $z$,
hence they should be qualified as {\em antiautomorphisms}.
However, since they exchange also $a$ and $b$, in \cite{BeKi10b},
they are qualified as automorphisms of the structure map $\Gamma$.
The true challenge, then,  is to find antiautomorphisms in the sense defined in loc.\ cit.
\end{remark}

\begin{remark}
For $L$- and   $R$-relations that are not invertible operators, the self-distribu\-tivity identity has to be replaced by
{\em structurality}, as defined in \cite{BeKi10}.
\end{remark}

\subsection{Canonical kernel, and structure of $U_{ab}$}
The following definition and lemma are purely set-theoretic:

\begin{definition}
For $x \in U_b$, $y \in U_a$, the {\em canonical kernel} is the map defined by
$$
B^{a,x,b}_y:= B^{axb}_y  : y \to y, \quad \eta \mapsto P^a_y \circ P^b_x (\eta) .
$$
\end{definition}

\begin{lemma}
Assume $a \top y \top b \top x$.
Then:
$B^{axb}_y:y\to y$ is bijective iff
$a \top x$.
If this is the case, then we have
$(B^{axb}_y)\inv = B^{bax}_y$.
\end{lemma}

\begin{proof}
From $y \top b \top x$ it follows that $P^b_x\vert_y : y \to x$ is bijective.
Therefore $B^{axb}_y$ is bijective iff $P^a_y\vert_x: x \to y$ is bijective, and this in turn is equivalent to saying
that $a \top x$.
If this holds, then, for all $\eta \in y$, we have
$B_y^{bxa} B_y^{axb}(\eta) = P^b_y P^a_x P^a_y P^b_x(\eta)= P^b_y P^a_x P^b_x(\eta)=P^b_y P^b_x(\eta)=P^b_y(\eta)=\eta$.
\end{proof}

\begin{theorem}\label{KernelTheorem}
The group $(U_{ab},y)$ has a natural action on the set $y$, given by 
$$
U_{ab} \times y \to y, \quad (x, \eta) \mapsto B^{a,x,b}_y (\eta) . 
$$
Thus  $B^{a,(xy z) ,b}_y = B^{a,x,b}_y \circ B^{a,z,b}_y$ and
$(B^{axb}_y)\inv = B^{a x\inv b}_y = B^{a (yxy) b}_y$.
\end{theorem}

\begin{proof}
Since $U_{ab}$ acts, as shown in the preceding section, on $\Omega$ by left translations via 
$(x,\zeta) \mapsto L_{xayb}(\zeta)$, it acts naturally on the {\em tautological bundle} of $\cP(\Omega)$,
which is the set defined by
$M:= \{ (z,\zeta) \mid z \in \cP(\Omega), \zeta \in z \}$
(the fiber over the ``point'' $z$ is the set of elements of $z$).
Similarly, $U_{ab}$ acts also by right translations, commuting with the left action, and hence there is also an action by conjugation,
$$
g.(z,\zeta) = ( L_{xayb} R_{yaxb} (z) , L_{xayb} R_{yaxb} (\zeta) ).
$$
The fiber over the neutral element $y$ is stable under the conjugation action, and thus we see that $U_{ab}$
acts on the set $y$ via
$
\eta \mapsto L_{xayb} R_{yaxb} (\eta) = R_{yaxb} L_{xayb}(\eta)
$.
But
$$
L_{xayb}(\eta) = [P^a_x P^b_y (\eta), P^b_y(\eta);\eta]_a = [P^a_x(\eta),\eta;\eta]= P^a_x(\eta),
$$
and similarly, for $\xi:= P^a_x(\eta)$, we get $R_{yaxb}(\xi)= P^b_y (\xi)$ and thus
$$
R_{yaxb} L_{xayb}(\eta) = 
R_{yaxb}(\xi)= P^b_y (\xi) = P^b_y P^a_x(\eta) = B^{a,x,b}_y(\eta) ,
$$
whence the claim.
\end{proof}

In the framework of associative geometries \cite{BeKi10}, more can be said about this action
(image, kernel, possible splittings). In a certain sense, the general situation should be a sort of amalgamation
of the following two extremal cases:

\begin{example}
If $a \top b$ and $y \in U_{ab}$, then the morphism
$(U_{ab} ,y) \to (\Bij(y), \id_y)$ is an isomorphism.
\end{example}

\begin{example}
If $a=b$, then $B^{axa}_y = \id_y$, for all $x \in U_{ab}$, so the morphism is trivial.
\end{example}

\section{Commuting prev's:  ``affine picture'',  and distributivity}\label{sec:distrib}

\subsection{Affine picture}
Assume $(\Omega,a,[\quad])$ is a (left) prev. Then $U_a$ carries a torsor structure (Corollary \ref{cor:groupUa}),
but we have more:
the structure of a  prev is considerably  stronger than the one of a torsor bundle, since all fibers are isomorphic to the model
torsor $G^a = \{ \lambda^a_{\xi,\eta} \mid (\xi,\eta) \in a \}$ of left translations (cf.\ Appendix \ref{App:B}).
This permits us to define a sort of coordinates on $U_a$, in a similar way as affine parts of projective spaces are
defined (so we call it an ``affine picture'' even in our general situation):

\begin{lemma}
Assume $(\Omega,a,[\quad])$ is a left prev, and fix a section $y \in U_a$ (one may call it ``zero section'').
Then the map
$$
\Map(y , G^a)  \to U_a, \quad f \mapsto z:= \{ f(\eta).\eta  \mid \, \eta  \in y  \} , 
$$
is bijective, with  inverse map
$$
U_a \to \Map(y,G^a), \quad
z \mapsto f = (\eta \mapsto  \lambda^a_{P^a_z(\eta),\eta} ).
$$ 
It intertwines the group structure on $(U_a,y)$ with the natural ``pointwise'' group structure of $\Map(y,G^a)$.
\end{lemma}

The proof is straightforward.
Note also that the prev-structure permits to identify a section $y\in U_a$ with an equivalence relation whose
equivalence classes are the sets $g.y$ for $g\in G^a$ (lines parallel to $y$, in the following illustration). 
In the following, to be in keeping with \cite{Be12,BeKi10}, we write the group law of $(U_a,o)$, for some origin $o$, additively
$x + z := (xoz)_{a}$ (which need not be commutative), so that
$
(uvw)_a = u -v +w $.

\subsection{Case of two commuting prev's}
From now on, assume that $(\Omega,a,b,[\quad])$ is an associoid given by 
 {\em two commuting principal equivalence relations (prev's)}
(Appendix \ref{App:C}).
This structure is stronger than the one of a pregroupoid: there are three para-associative laws defined on $\cP(\Omega)$,
denoted by $(\quad)_{ab}$, $(\quad)_a$ and $(\quad)_b$, and one may ask how these maps interact.
We will show that they are related by a certain {\em distributive law}, by using the following description of 
the product $(xyz)_{ab}$ in the ``affine picture'':

\begin{theorem}\label{ActionTheorem}
Let $(a,b)$ be two commuting prev's on $\Omega$, and fix a bisection $y \in U_{ab}$, considered as zero section in $U_a$,
write $u+v:=(uyv)$ for the group law in $U_a$ and $x \cdot z := (xyz)$ for $x,z \in U_a$. Then, identifying $U_a$ with
$\Map(y,G^a)$ as above, we have, for all $x,z \in \Map(y,G^a)$, 
$$
x \cdot z = x \circ B^{bza}_y + z .  
$$
\end{theorem}

\begin{proof}
The theorem and its proof follow closely the corresponding result in the group case
(\cite{Be12}, Th. 8.1; cf.\ figure below).
The element of  $\Map(y,G^a)$  corresponding to $x \cdot z$ is
$$
y \to G^a, \, 
\eta \mapsto \lambda^a_{\omega,\eta}, \quad \mbox{ where }
\omega = P^a_{x\cdot z}(\eta) .
$$
Let $\zeta :=  P^a_z(\eta)$.
Then $\omega = [P^a_x P^b_y(\zeta) \, P^b_y(\zeta) \,  \zeta ]$: this element belongs to $x\cdot z$ and is 
$a$-related to $\zeta$ and to $\eta$. Thus, letting
$\eta' := P^b_y (\zeta) = P^b_y P^a_z(\eta) = B_y^{bza}(\eta)$, $\xi:= P^a_x(\eta') = P^a_x P^b_y(\zeta)=
P^a_x B^{bza}_y(\eta) $, we have
 $\omega = 
 [\xi \eta' \zeta] $.
 Then we have
 $\lambda^a_{\omega,\zeta}= \lambda^a_{\xi,\eta'}$,  so
 $$
 \lambda^a_{\omega,\eta} = \lambda^a_{\omega,\zeta} \circ \lambda^a_{\zeta,\eta} =
  \lambda^a_{\xi,\eta'} \circ \lambda^a_{\zeta,\eta} =
  \lambda^a_{P^a_x B^{bza}_y(\eta) ,B_y^{azb}(\eta)} \circ \lambda^a_{\zeta,\eta} ,
 $$ 
which, written additively, corresponds to  the formula given in the claim.
\end{proof}

Recall from \cite{Be12} that the formula given in the theorem is the non-linear generalization of
bilinear formulae from associative and Jordan algebraic setting.
The proof gives  in fact an algorithm permitting to draw (using dynamical geometry software, such
as {\tt geogebra})
 the curve $x \cdot z$ as a function of
$\eta$, as in the following illustration:

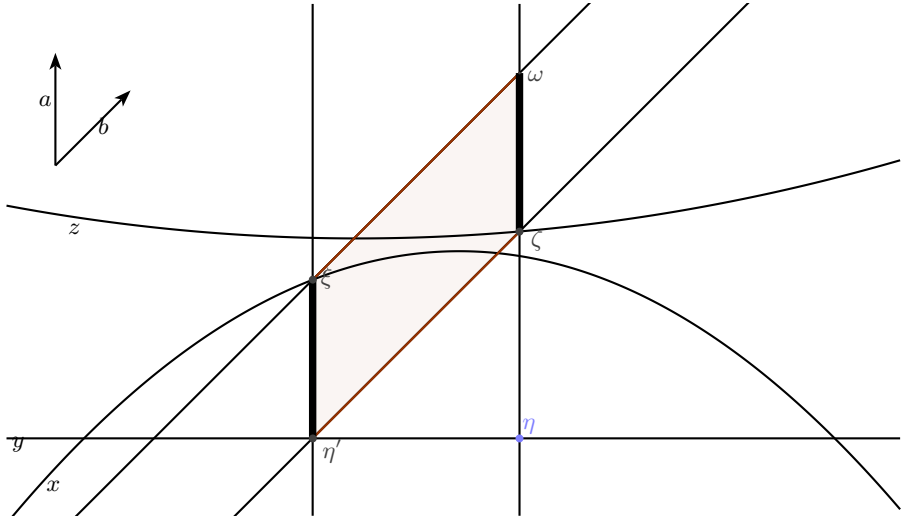
\begin{center}
\newrgbcolor{xdxdff}{0.49 0.49 1}
\newrgbcolor{zzttqq}{0.6 0.2 0}
\psset{xunit=0.5cm,yunit=0.5cm,algebraic=true,dotstyle=o,dotsize=3pt 0,linewidth=0.8pt,arrowsize=3pt 2,arrowinset=0.25}
\begin{pspicture*}(-4.3,-7.32)(19.46,6.3)
\pspolygon[linecolor=zzttqq,fillcolor=zzttqq,fillstyle=solid,opacity=0.05](9.34,0.24)(9.34,4.46)(3.84,-1.04)(3.84,-5.26)
\psline{->}(-3,2)(-3,5)
\psline{->}(-3,2)(-1,4)
\psplot[plotpoints=200]{-4.3}{19.46}{0.01*(x-5.04)^2+0.06}
\psplot[plotpoints=200]{-4.3}{19.46}{-2-0.05*(x-7.74)^2+1.72}
\psplot[plotpoints=200]{-4.3}{19.46}{-3-2.26}
\psline(9.34,-7.32)(9.34,6.3)
\psplot{-4.3}{19.46}{(-18.19--2*x)/2}
\psline(3.84,-7.32)(3.84,6.3)
\psplot{-4.3}{19.46}{(-9.76--2*x)/2}
\psline[linecolor=zzttqq](9.34,0.24)(9.34,4.46)
\psline[linecolor=zzttqq](9.34,4.46)(3.84,-1.04)
\psline[linecolor=zzttqq](3.84,-1.04)(3.84,-5.26)
\psline[linecolor=zzttqq](3.84,-5.26)(9.34,0.24)
\psline[linewidth=2.8pt](9.34,0.24)(9.34,4.46)
\psline[linewidth=2.8pt](3.84,-1.04)(3.84,-5.26)
\begin{scriptsize}
\rput[bl](-3.44,3.58){$a$}
\rput[bl](-1.86,2.84){$b$}
\rput[bl](-2.66,0.18){$z$}
\rput[bl](-3.24,-6.66){$x$}
\rput[bl](-4.16,-5.62){$y$}
\psdots[dotstyle=*,linecolor=xdxdff](9.34,-5.26)
\rput[bl](9.42,-5.14){\xdxdff{$\eta$}}
\psdots[dotstyle=*,linecolor=darkgray](9.34,0.24)
\rput[bl](9.64,-0.26){\darkgray{$\zeta$}}
\psdots[dotstyle=*,linecolor=darkgray](3.84,-5.26)
\rput[bl](4.1,-5.92){\darkgray{$\eta'$}}
\psdots[dotstyle=*,linecolor=darkgray](3.84,-1.04)
\rput[bl](4.06,-1.28){\darkgray{$\xi$}}
\psdots[dotsize=1pt 0,dotstyle=*,linecolor=darkgray](9.34,4.46)
\rput[bl](9.56,4.24){\darkgray{$\omega$}}
\end{scriptsize}
\end{pspicture*}
\end{center}

%

\begin{theorem}\label{DistributiveTheorem}
Assume $a$ is a left prev and $b$ a right prev on $\Omega$ commuting with $a$.
Then we have the following
``right  distributive  law'': for all $x,y \in U_{ab}$ and $u,v,w \in U_b$,
$$
( (uvw)_a yz )_{ab} = \bigl( (uyz)_{ab} (vyz)_{ab} (wyz)_{ab} )_{a} ,
$$
which can also be written, with notation introduced above,
$$
((u-v+w)yz)_{ab} = ( uyz)_{ab}  - ( vyz)_{ab} + ( wyz)_{ab} .
$$
In other words, right  multiplications operators $R_{zbya}$ from $U_{ab}$ are automorphisms of the torsor
$U_{a}$.
Similarly,  left  multiplications  from $U_{ab}$ are automorphisms of
the torsor $U_b$.
\end{theorem}

\begin{proof} By the preceding result
\begin{eqnarray*}
(u-v+w) \cdot z  &=& (u-v+w) \circ B_y^{bza} + z 
\cr
& =& (u \circ B_y^{bza} + z ) - (v \circ B_y^{bza} + z) + (w \circ B_y^{bza} + z) \cr
&=&  (u \cdot z) - (v \cdot z) + (w \cdot z) 
\end{eqnarray*}
(see  \cite{Be12}, Th.\ 8.3,  for the -- similar -- proof in the group case). 
\end{proof}

Note that, in general, $U_a$ will not be left-distributive with respect to $(\quad)_{ab}$.
In this respect, this structure resembles the one of a {\em near-ring}, cf.\ \cite{Pi77} (cf.\  \cite{Be12}).

\section{The homogeneous  case, and duality}\label{sec:groupcase}

\subsection{The group case}
Assume $(\Omega,+)$ is a group and $A,B$ are two subgroups of $\Omega$.
Then $(\Omega,a,b,[\quad])$ as defined in example \ref{ex:homog!} is a pregroupoid, coming from two
commuting prev's $(a,b)$.
Thus all  of the preceding results apply (see  \cite{Be12} for more details).
The description of the set $(xyz)_{ab}$ from Theorem \ref{th:semitorsor2}  can be rewritten as in 
Equation (\ref{eqn:str}). The system of three equations given in (\ref{eqn:str}),
\begin{equation}\label{eqn:str2}
\eta = \alpha + \xi, \qquad
\eta = \zeta + \beta, \qquad
\omega = \xi - \eta + \zeta,
\end{equation}
is called the {\em structure equations}.  As a specific feature of
 the group case,  the structure equations can be rewritten in various equivalent ways, see 
 \cite{Be12}, Lemma 2.3.  The form given in (\ref{eqn:str2}) is closest to the interpretation 
 used in this work; but there are other forms suggesting other interpretations, invoquing  some sign changes
 (recorded by the ``sign-vectors'' listed in section 9 of \cite{Be12}).
 One may distinguish two different guises of the structure equations, having a quite different flavor:
 
 \subsection{Duality, and antiautomorphisms}
 On the one hand, there is a number of ways to rewrite (\ref{eqn:str2}) in the form of $3$ equations, two of them
 invoking $3$ variables and the last one invoking $4$ variables.
 These equations all ``have the same shape'', up to the possible sign changes (collected by the sign-vectors),
 and thus take account of a high degree of symmetry in the group case, 
  including the possibility of
 {\em antiautomorphisms} -- recall from \cite{BeKi10b} that, for abelian groups $(\Omega,+)$,
 orthocomplementation maps are indeed antiautomorphisms in this sense (and so far these are the
 only examples of antiautomorphisms we know of).

 \subsection{Triangle configurations}
 The structure equations can also be rewritten as systems of three equations, each of them
 invoking $3$ variables, e.g., 
\begin{equation}\label{EquationS8}
\left\{
\begin{matrix}
\eta  &=& \alpha + \xi \cr
\beta  &=& - \omega + \xi \cr
\zeta &=& \alpha + \omega
\end{matrix}
\right\}
\qquad
\left\{
\begin{matrix}
\alpha &=& \eta - \xi \cr
\zeta  &=& \eta - \beta \cr
\omega  &=& \xi - \beta
\end{matrix}
\right\}
 \qquad
\left\{
\begin{matrix}
\alpha &=& \zeta - \omega \cr
\xi   &=& \omega +  \beta \cr
\eta  &=& \zeta + \beta
\end{matrix}
\right\}
\end{equation}
This form of the equations suggests an entirely different interpretation of the system:
the variables appearing just once shall correspond to equivalence relations
(given by cosets) and those appearing twice correspond to sets whose elements are in the respective relation.
Thus we deal with {\em three} relations. Graphically, we have {\em triangles}, instead of parallelograms.
Neglecting the sign changes,
the following figures would then correspond to the three systems from (\ref{EquationS8}):

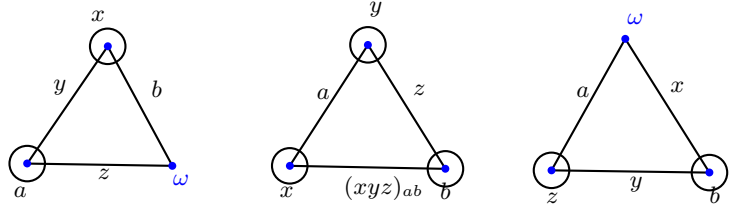
\begin{center}
\psset{xunit=0.5cm,yunit=0.5cm,algebraic=true,dotstyle=o,dotsize=3pt 0,linewidth=0.8pt,arrowsize=3pt 2,arrowinset=0.25}
\begin{pspicture*}(-2,-3.5)(18,3)
\psline(0.72,1.18)(-1.42,-1.94)
\psline(-1.42,-1.94)(2.44,-2)
\psline(0.72,1.18)(2.44,-2)
\psline(5.56,-2)(7.64,1.22)
\psline(7.64,1.22)(9.7,-2.08)
\psline(5.56,-2)(9.7,-2.08)
\psline(14.48,1.38)(12.52,-2.12)
\psline(12.52,-2.12)(16.72,-2.18)
\psline(14.48,1.38)(16.72,-2.18)
\pscircle(0.72,1.18){0.25}
\pscircle(-1.42,-1.94){0.25}
\pscircle(7.64,1.22){0.25}
\pscircle(5.56,-2){0.25}
\pscircle(9.7,-2.08){0.25}
\pscircle(12.52,-2.12){0.25}
\pscircle(16.72,-2.18){0.25}
\begin{scriptsize}
\psdots[dotstyle=*,linecolor=blue](0.72,1.18)
\psdots[dotstyle=*,linecolor=blue](-1.42,-1.94)
\rput[bl](-0.72,-0.12){$y$}
\psdots[dotstyle=*,linecolor=blue](2.44,-2)
\rput[bl](2.46,-2.5){\blue{$\omega$}}
\rput[bl](0.46,-2.36){$z$}
\rput[bl](1.88,-0.18){$b$}
\psdots[dotstyle=*,linecolor=blue](5.56,-2)
\psdots[dotstyle=*,linecolor=blue](7.64,1.22)
\rput[bl](6.28,-0.22){$a$}
\psdots[dotstyle=*,linecolor=blue](9.7,-2.08)
\rput[bl](8.84,-0.12){$z$}
\rput[bl](7.00,-2.9){$(xyz)_{ab}$}
\psdots[dotstyle=*,linecolor=blue](14.48,1.38)
\rput[bl](14.54,1.78){\blue{$\omega$}}
\psdots[dotstyle=*,linecolor=blue](12.52,-2.12)
\rput[bl](13.18,-0.16){$a$}
\psdots[dotstyle=*,linecolor=blue](16.72,-2.18)
\rput[bl](14.6,-2.7){$y$}
\rput[bl](15.68,-0.04){$x$}
\rput[bl](0.28,1.86){$x$}
\rput[bl](-1.78,-2.82){$a$}
\rput[bl](7.68,1.96){$y$}
\rput[bl](5.3,-2.82){$x$}
\rput[bl](9.56,-2.88){$b$}
\rput[bl](12.38,-3){$z$}
\rput[bl](16.72,-3){$b$}
\end{scriptsize}
\end{pspicture*}
\end{center}

\nin
Unlike the interpretation invoking parallelograms, these  triangle configurations do not need at all the underlying
group or pregroupoid structure of $\Omega$: they make sense  on a ``naked set''. Thus
such interpretation seems even more basic  than the preceding one.
Indeed, whereas for usual mappings there is essentially just one ``type of commuting triangles'', 
for binary relations the term 
``commuting triangle'' may have several different meanings, and it looks as if the structure equations were related to this.
To our knowledge, such questions  have so far only been investigated from the point of view of computer
science, see e.g., the ``diagram chase for relations'', 
 \cite{Struth}.
It would certainly be worth investigating  this from a more theoretical and algebraic point of view.


\appendix

\section{The familily of associoids}\label{APP}

We assume that the reader is familiar with the most important binary associative structures:
{\em (semi-)groups, monoids, groupoids, algebras}. In this appendix we give definitions and fix terminology concerning related
{\em ternary associative structures}. Since their laws need not be defined everywhere, we use the suffix -oid,
and call {\em associoids} the most general family of such structures.
They arise naturally in various contexts, but often are are 
not recognized as such (e.g.,  
principal bundles are rarely seen as ternary algebraic structures).

\subsection{Associoids and semi-associoids}

\begin{definition}\label{def:ternaryproduct}
Let $M$ be a set.
A {\em partially defined ternary product on $M$} is a map,
defined on some non-empty subset $D \subset M^3$ (called {\em domain} of $\mu$),
$$
\mu: D \to M, \quad (x,y,z) \mapsto \mu(x,y,z)
$$
According to context,  we use other notation for $\mu(x,y,z)$, such as $(xyz)$ or 
$[xyz]$.
\end{definition}

\begin{definition} 
A {\em semi-associoid} is a set $M$ together with
a partially defined ternary product $(xyz)=\mu(x,y,z)$ on $M$ satisfying  the {\em para-associative law}:
\[
(xy(zuv))= (x(uzy)v)=((xyz)uv)\,   \tag{PA}
\]
that is, if all terms on one side of the equation are defined, then so are those on the other sides,
and we have equality. For instance,  if $(z,u,v) \in D$ and  $(x,y,(zuv)) \in D$, 
then $(u,z,y) \in D$ and $(x,(uzy),y) \in D$, and the first equality holds, 
and so on.
An {\em associoid} is a semi-associoid satisfying  the 
\emph{idempotent law}: 
\[
(xxy)=y, \quad (wzz)=w  \, , \tag{IP}
\]
whenever $(x,x,y) \in D$ and $(w,z,z) \in D$. 
\end{definition}

\begin{lemma}
The associoid axioms  (PA) $\land$ (IP) are  equivalent to (Ch) $\land$ (IP):
\begin{enumerate}
\item[(Ch)]
left Chasles relation: $(xy(yuv))=(xuv)$, and
\\
right Chasles relation: $((xyz)zv)=(xyv)$;
\item[(IP)] idempotency:
$(xxy)=y=(yxx)$ 
\end{enumerate}
\end{lemma}

\begin{proof}
 (IP) $\land$ (Ch) implies (PA):
$(xy(uvw)) = ((xyu)u(uvw))=((xyu)vw)$, and conversly
(IP) $\land$ (PA)  implies  (Ch) by taking $y=z$.
\end{proof}

\nin
In his work \cite{Ko82, Ko05, Ko07}, Kock uses (IP) (``unit law'') $\land$ (Ch) (``cancellation law'') as basic axioms;
note, however, that  the  lemma  does not carry over to semi-associoids.
There is a number of other useful identities valid in general associoids, most of them well-known for torsors (see below), for instance,
\begin{equation}\label{eq:autos}
(xy(uvw)) = ((xyu)vw)= ((xyu) (yx(xyv)) w) 
 =  ((xyu) (xyv) (xyw)) , 
\end{equation}
which, in case the product is defined everywhere, means that left translations $\ell_{xy}$ are endomorphisms.
Similarly, right translations satisfy the endomorphism property.

\subsection{Torsors and semitorsors} \label{App:A}
This is the case of {\em everywhere defined} products:

\begin{definition}
\label{TorsorDefinition}
A {\em semitorsor} is a semi-associoid   with an everywhere defined ternary product map
(often denoted by $G^3 \to G$, $(x,y,z) \mapsto (xyz)$), and
a  {\em torsor} is a semitorsor satisfying the idempotent law.
\footnote{
Other terms used in the  literature instead of  ``torsor'' are {\em groud},
{\em heap}, {\em pregroup}, {\em herd}  or  {\em principal homogeneous space}.} 
\end{definition}

Fixing the middle element $y$ in a torsor $G$, we get a group law
$xz:=(xyz)$ with neutral element $y$ and inversion $x\inv = (yxy)$, and every group is obtained in this way;
thus torsors are for groups what affine spaces are for vector spaces.
Similarly, semitorsors give rise to semigroups, but the converse is more
involved. 
Two kinds of examples will play a basic role in this work:

\begin{example} [Relation semitorsors and mapping torsors]\label{ex:RelTor}
The set
 $M:=\cR(\Omega',\Omega)$ of binary relations between $\Omega$ and $\Omega'$ is a semitorsor:
 the everywhere defined ternary product
$[RST] := RS\inv T$ is  para-associative, by associativity of binary composition and involutivity of reversion.
Likewise,  the set $\Bij(X,Y)$ of bijections between two sets $X$ and $Y$,
with the law $(xyz)=xy\inv z$, is a torsor. 
\end{example}

\begin{example}[Torsors of sections]
The set of sections of a principal bundle is a torsor, see cor.\  \ref{cor:groupUa}.
\end{example}

 In every semitorsor, we introduce 
{\em left-, right-} and {\em middle multiplication operators}
\begin{equation}
\ell_{x,y}(z) = (xyz) = r_{z,y}(x) = m_{x,z}(y).
\end{equation} 
These are everywhere defined operators $\ell_{x,y}:M\to M$, etc., and
all defining identitities may be rewritten in terms of these operators. In particular,
 the   left Chasles relation reads  $\ell_{ab}\ell_{bc}=\ell_{ac}$, and the right Chasles relation becomes
 the ``transplantation formula''
\begin{equation}
\ell_{xy}= \ell_{(xyz)z}.
\end{equation}
See appendices of \cite{BeKi10, Be13} for some more remarks on general (semi)torsors.

\subsection{Pregroupoids and semi-pregroupoids}
From now on, the domain of definition $D \subset M^3$ of the ternary product  will depend on the choice of a pair of
equivalence relations $(a,b)$ on $M$. We symbolize this situation by
\begin{equation}
\begin{matrix}
& & M & & \cr & \swarrow & & \searrow & \cr M/a & & & & M/b 
\end{matrix}
\end{equation}
Consider the following four sets
\begin{eqnarray}\label{eqn:Domains}
 M\times_a M\times M  &=& a \times M  =   \{ (x,y,z) \in M^3 \mid x \sim_a y  \}
 \cr
 M\times M\times_b  M  &=&
 M \times b  =  \{ (x,y,z) \in M^3 \mid   y\sim_b z \}
 \cr
 M\times_a M\times_b M & =&
 (a \times M) \cap (M \times b) 
 \cr
 & = & \{ (x,y,z) \in M^3 \mid x \sim_a y \mbox{ and }  y\sim_b z \}
\cr 
(a \times M ) \cup (M \times b) &= & \{ (x,y,z) \in M^3  \mid x \sim_a y
\mbox{ or }  y\sim_b z \} .
\end{eqnarray}

\begin{definition}\label{def:pregroupoid}
A {\em semi-pregroupoid} is given by a set $M$ together with a pair $(a,b)$ of  equivalence relations
on $M$ and a para-associative ternary multiplication
$$
M \times_a M \times_b M 
 \to M, \quad
(x,y,z) \mapsto [xyz] 
$$
such that, for all $(x,y,z) \in M\times_a M \times_b M$, 
\[
[xyz] \sim_a z, \qquad [xyz] \sim_b x .
\tag{C}
\]
If, moreover, the multiplication satisfies the idempotent law, then $M$ is called a {\em pregroupoid}.\footnote{This terminology
is due to Kock, see \cite{Ko05, Ko07}. Johnstone uses {\em herdoid}, see \cite{Jo91}.}
\end{definition}

\nin
Note that Condition (C) ensures that our requirement on the domain $D$ from Definition \ref{def:ternaryproduct}  is fulfilled.

\begin{lemma}\label{la:abcommute}
If $(M,a,b,[\quad])$ is a semi-pregroupoid, then $a$ and $b$ commute as relations: $ab=ba$.
\end{lemma}

\begin{proof}
Let $(xz) \in ab$, so there is $y \in M$ with $(x,y) \in a$, $(y,z) \in b$.
Thus $w:=[xyz]$ is defined, and by (C), we have
$(x,w)\in b$, $(w,z) \in a$, showing that $(x,z) \in ba$.
\end{proof}

\nin
The condition $ab=ba$   may be seen as a sort of ``integrability condition'' -- indeed, 
it is well-known (and easy to prove) that $a$ and $b$ commute iff the new relation
$ab$ is again an equivalence relation.

\begin{example}[The relation semi-pregroupoid]
Let $M:=\cR(\Omega',\Omega)$ be the set of binary relations between $\Omega$ and $\Omega'$.
For two relations $R,S \in \cR(\Omega',\Omega)$ write
$$
R \sim_a S \, :\Leftrightarrow \, \dom(R) = \dom(S), \qquad
R \sim_b S \, :\Leftrightarrow \, \im(R) = \im(S) .
$$
The para-associative law for the ternary product
$[RST] := RS\inv T$ is always satisfied (example \ref{ex:RelTor} above).
Moreover, it follows directly from definitions that its restriction to 
$(a\times M) \cap ( M \times b)$ satisfies condition (C). The idempotent law is not satisfied.
\end{example}

\begin{example}[Pair semi-pregroupoid]\label{ex:pairpregroupoid}
Assume $M = E \times F$, $a$, $b$ the equivalence relations given by fibers of $\pr_1,\pr_2$,
so
$$
M\times_a M \times_b M = \{ (x,y,z) \in M^3 \mid x_1 = y_1, y_2 = z_2 \}
$$
and then the algebraic law is already completely determined by Condition (C):
$$
[xyz ] = (z_1,x_2) .
$$
Put differently: whenever $a \top b$, there is just one possible semi-pregroupoid structure on $(M,a,b)$
(which is in fact a pregroupoid structure). 
It generalizes the {\em pair groupoid} (case $E=F$, cf.\ \cite{CW}), and we will use the same term.
As for the case of groupoids, for any (semi-) pregroupoid, the canonical map
$$
M \to M/a \times M/b, \quad x \mapsto ([x]_a,[x]_b)
$$
is a morphism of (semi-) pregroupoids onto a pair groupoid (see \cite{CW}, 13.2). 
\end{example}

\begin{example}[Pseudogroup of local bijections]
Let $M$ be the set of locally defined bijections between $\Omega$ and $\Omega'$, that is, 
$$
M=\{ R \in \cR(\Omega',\Omega) \mid \, \forall x \in \dom(R) : \exists^! y \in \im(R) : (y,x) \in R \} 
$$
Then $M$ is sub-semi-pregroupoid of the one given in the preceding example, and it satisfies the idempotent law,
hence is a pregroupoid. 
In a context of manifolds and smooth maps, it corresponds to what one calls a {\em pseudogroup of local
diffeomorphisms}.
\end{example}

\begin{example}[Homogeneous pregroupoids]\label{ex:homog!}
Let $G$ be a group (not assumed commutative, but written additively) and
$A,B$ two subgroups of $G$.
Define the {\em right equivalence relation}
$x \sim_a y$  iff $A+x=A+y$  (so equivalence classes are right  cosets of $A$),
and the {\em left  equivalence relation}
$x \sim_b y$ iff $x+B=y+B$ (so equivalence classes are left  cosets of $B$),
and let
$[xyz] = x - y +z$
(which is the usual torsor law of $G$, and thus is para-associative and idempotent).
Then $(G,a,b,[\quad])$ is a pregroupoid which we denote by
\begin{equation}
\begin{matrix}
G/B & \leftarrow & G & \rightarrow A \backslash G 
\end{matrix} .
\end{equation}
Indeed, the compatibilty condition (C) is easily checked.
One should note that the structure of $(G,A,B,[\quad])$ is in fact much richer: it is  a 
pregroupoid with a certain additional structure, which is witnessed by the term ``homogeneous''.
(A homogeneous pregroupoid is transitive in the obvious sense, see subsection \ref{ssec:trans};
and a transitive pregroupoid is  a quotient of some $(G,A,B)$, where $G$ is a, in  general not uniquely
determined, group coming from transitive structure.)
Note also that this pregroupoid structure does in general not come from a groupoid structure:
the quotient sets $G/B$ and $A \backslash G$ need not be isomorphic as sets.
\end{example}

\begin{example}
If $a=b$, then a pregroupoid $(\Omega,a,a,[\quad])$ is the same thing as a {\em torsor bundle over $\Omega/a$}.
A torsor bundle together with some fixed section is a {\em group bundle}.
In particular, {\em vector bundles} are a special kind of associoids. 
\end{example}

\subsection{Principal equivalence relations (prev's)}\label{App:B}
This is an algebraic concept corresponding to the one of {\em principal bundle}:

\begin{definition}
A {\em (left)  principal equivalence relation} (abbreviated: {\em (left) prev})
$(M,a,[\quad])$ is an associoid $(M,D,[\quad])$ with domain
$D = a \times M$ given by an equivalence relation $a$ on $M$ and such that
$$
\forall (x,y)\in a, \forall z \in M: \qquad ([xyz],z) \in a .
$$
In other words, it is a pregroupoid defined by $(a,b)$ with $b = \All_M$.
Similarly, 
 a {\em right principal equivalence relation}  
on a set $M$ is an associoid with domain $M \times b$, that is, a
pregroupoid with $a = \All_M$.
\footnote{Once more, there may be  conflicts of terminology: 
what is called ``torseur'' in \cite{DG}  corresponds to what we would call a
``prev in the category of schemes''. }
\end{definition}

\nin 
The good thing about left prev's is that the {\em left translation operators}
$\lambda^a_{xy}:z \mapsto [xyz]$ are everywhere defined.
We often write $[xy;z]_a:=[xtz]$ and view  a left prev   as an equivalence relation
$a$ together with a ``left action map'' 
$$
\lambda^a : a \times M  \to M, \quad ((\xi,\eta),\zeta) \mapsto \lambda^a_{\xi \eta}(\zeta) =: [\xi \eta;\zeta]_a 
$$
such that

\begin{enumerate}
\item
$\lambda^a_{xy}$ preserves each $a$-equivalence class, 
that is, $a \circ \lambda^a_{xy}=a=\lambda^a_{xy}\circ a$ under relational composition,
\item
 for all $(x,y)\in a$ and $(v,w) \in a$, we have 
 
 $\lambda^a_{xx} = \id_\Omega$, $\lambda^a_{xy}(y)=x$,

$
\lambda_{xy} \circ  \lambda _{vw} = \lambda_{[xyv],w} = \lambda_{x,[wvy]} ,
$
\end{enumerate}

\nin which imply {\em Chasles' relation} $\lambda_{xy} \circ \lambda_{yz}=\lambda_{xz}$
and $(\lambda_{xy})\inv = \lambda_{yx}$ as well as
\begin{equation}
\lambda^a_{xy} = \lambda^a_{u,\lambda^a_{xy}(u)}.
\end{equation}
Since each equivalence class of $a$ is a torsor, it follows that the {\em left translation group}
\begin{equation}
G:=G^a := \{ \lambda^a_{xy} \mid \, (x,y) \in a \}
\end{equation}
acts simply transitively on each $a$-class. 
Thus $(M,G^a,M/a)$ is an ``abstract principal bundle''.

\begin{lemma} 
There is a bijection between
left prev's $(M,a,[\quad]_a)$ on $M$ and (abstract) left principal bundles $(P,G,B)$ with total space $P=M$.
 \end{lemma}

\begin{proof}
One direction is explained above; conversely, given 
$(P,G,B)$, let $a$ be the fiber relation of the canonical projection $P \to B$
and define, for $(x,y) \in a$ and  $z \in P$,
$[xy;z] := \lambda_{xy}(z)$, where
$g:=\lambda_{xy} \in G$ is the unique element such that $g(y)=x$. 
\end{proof}

Similarly, a right prev on $M$ is
 an equivalence relation
$b \in \cE_\Omega(M)$ together with a 
 ``right action map'' satisfying the para-associative and idempotent laws,  
$$
\rho^b : \Omega \times b  \to \Omega, \quad (x,(y,z)) \mapsto \rho^b_{zy}(x) =: [x;yz]_b .
$$

\begin{example} [The homogeneous pregroupoid revisited]
Let $\Omega = G$ and $A,B,a,b$ as in example \ref{ex:homog!}.
Then $G \to G/B$ and $G\to A \backslash G$ are principal bundles, corresponding to left- and
right prevs given by $[xy;z]_a = x-y+z = [x;yz]_b$.
\end{example}

\subsection{Commuting left and right prev's}\label{App:C}

\begin{definition}
A {\em pair $(a,b)$ of commuting left and right prev's} is an associoid given by a pair $(a,b)$ of 
equivalence relations on $M$ and the ternary
product $[xyz]$ defined on the domain
$$
D_{a,b}:= (a \times M)\cup (M \times b) 
$$
such that
$$
\forall (x,y) \in a, \forall z \in M: \quad [xyz] \sim_a z, 
\qquad
\forall x \in M, \forall (y,z) \in b: \quad [xyz] \sim_b x.
$$
In other words, $(M,a,[xy;z]_a)$ is a left prev and $(M,b,[x;yz]_b)$ a right prev such that the natural
compatibility conditions given below are satisfied. 
\end{definition}

\nin
Since $(a \times M)\cap (M \times b) \subset (a \times M)\cup (M \times b)$,
it is clear that a pair of commuting prev's defines, by restriction of the product map to a smaller domain, a pregroupoid (but not
every pregroupoid is of this form). The definition may be stated
equivalently: a pair of commuting prev's is given by
a left prev $a$ and a right prev $b$ on $\Omega$ that are  {\em compatible} in the sense that
$$
 (x,y)\in a, (y,z) \in b  \quad \Rightarrow  \quad [xy;z]_a = [x;yz]_b ,
$$
and which, moreover,  {\em commute} in the sense that
$$
\forall (x,y) \in a, (w,z) \in b : \qquad 
\lambda^a_{xy} \circ \rho^b_{uv} = \rho^b_{uv}\circ  \lambda^a_{xy} .
$$
The last condition can be rewritten
 $[xy;[v;wz]_b]_a = [[xy;v]_a;wz]_b $, for all $v \in \Omega$.
By compatibility, it  implies, if moreover $v \sim_a w$ and $v \sim_b y$, 
$$
[xy;[v;wz]_b]_a = [x[w;vy]_b;z]_a = [x[w;vy]_b; z]_a  .
$$

\begin{lemma} \label{la:lrprev}
If $(a,b)$ is a pair of commuting left and right prev's,  then
$a$ and $b$ commute as relations:   $ab=ba$, and 
 left translation operators $\lambda_{xy}^a$ permute equivalence classes of $b$, and vice versa:
$$
x\sim_a y, \,  u \sim_b v \quad  \Rightarrow \quad 
 [xy;u]_a \sim_b [xy;v]_a \mbox{ and }  [x;uv]_b \sim_a [y;uv]_b .
$$
\end{lemma}

\begin{proof}
That $ab=ba$ has been seen above (Lemma \ref{la:abcommute}).
The other claim follows from 
 $\lambda^a_{xy}(v) = \lambda^a_{xy} \rho^b_{vu}(u) =
\rho^b_{vu} \lambda^a_{xy} (u) \sim_b \lambda^a_{xy}(u)$.
\end{proof}

\begin{definition}\label{def:Auta}
The {\em automorphism group of an equivalence relation $a$} is the group
permuting equivalence classes of $a$:
$$
\Aut(a) = \{ g \in \Bij(M) \mid \, g \circ a = a \circ g \} 
$$
(so that the induced map $[g]:M/a \to M/a$ is well-defined), and
the {\em group of strict automorphisms of $a$} is the group preserving each equivalence class of $a$:
$$
\Aut_1(a) = \{ g \in \Bij(M) \mid \, g \circ a = a = a \circ g \}
$$
(so that the induced map is $\id_{M/a}$).
\end{definition}

Thus by the lemma we have $\lambda^a_{xy} \in \Aut_1(a) \cap \Aut(b)$ and
$\rho^b_{uv} \in \Aut_1(b) \cap \Aut(a)$.

 \begin{example} [Homogeneous pregrouids]
 Let $A,B$ be two subgroups of a group $\Omega$.
 Then the left prev $b$ defined by $B$ commutes with the right  prev $a$ defined by $A$. 
 Note that $ab=ba$ then is again an equivalence relation (its equivalence classes 
 are the double cosets $A\backslash \Omega /B$), but it is not a prev.
 \end{example}

\subsection{From pregroupoids to groupoids: back and forth}
A {\em groupoid} is a pregroupoid $(M,a,b)$ such that $M/a$ and $M/b$ are isomorphic and having a distinguished {\em set of units}
and an {\em inversion map}, and the whole object is encoded by a {\em binary} partially defined operation,
rather than by a ternary one. It is a remarkable fact that such a structure arises simply by 
distinguishing an arbitrary {\em bisection} in a pregroupoid, which defines the set of units.
The following lemma is contained in the work of Johnstone \cite{Jo91}, although it is not formally stated 
there.\footnote{loc.\ cit., p.\ 103: ``Conversely, given a herdoid for which $A$ and $B$ happen to coincide, any
choice of simultaneous splitting for $\alpha$ and $\beta$ (if such a thing exists) equips it with a groupoid structure.''}

\begin{definition}
A {\em bisection} of a pregroupoid $(M,a,b)$ is a bisection of $(a,b)$, i.e., 
 an element of $U_{ab}$, and a {\em local bisection} is an element of $U_{ab}^\loc$  (cf.\ eqn.\  (\ref{eqn:Uabloc})).
\end{definition}

\nin
In general, a pregroupoid does not admit bisections (since $M/a$ and $M/b$ need not be equinumerous).
For a fixed bisection $s$, and for any $g \in M$, we let $a_g:= [g]_a \cap s$ and $b_g := [g]_b \cap s$.
Then we define {\em the inverse of $g$ (with respect to $s$)} by 
\begin{equation}
g\inv := [a_g \, g \, b_g],
\end{equation}
and, if  $b_g = a_h$, then we define the {\em (binary) composition of $g$ and $h$} by
\begin{equation}
g \circ h := [g b_g h] = [g a_h h ].
\end{equation}

\begin{lemma}
Let $(M,a,b)$ be a pregroupoid with a distinguished bisection $s$.
Then $M$ with inversion and partially defined binary composition as above becomes a groupoid
with set of units $s$. Conversely, every groupoid arises in this way from a pregroupoid.
Thus 
groupoids are the same as as pregroupoids $(M,a,b)$, together with some fixed bisection $s$,  representing 
the  set of units.
\end{lemma}

\begin{proof}
$f(gh) = [f a_f [g a_h h]] = [[f a_f g] a_h h] = (fg)h$ follows immediately from para-associativity,
$g a_g = [g a_a a_g]=g$ and $b_g g = [b_g b_g g]= g$ from idempotency,
and
$g \circ g\inv = [ga_g [a_g g b_g]] = [[g a_g a_g] g b_g] = [gg b_g] = b_g$
and $g\inv \circ g = a_g$ from para-associativity and idempotency.
\end{proof}

\nin
In chapter \ref{sec:Bisections} of the present work, the lemma appears as special case of more general results.
Here is an illustration, following Cannas da Silva and Weinstein \cite{CW}.
We prefer to draw $s$ as a curved line, instead of a straight one, in order to stress that, in principle, any bisection
can serve as set of units.
Note that $(x,y\inv,z,xyz)$ forms a parallelogram, and given $x,y\inv,z$, the last vertex does not depend on the
bisection $s$:

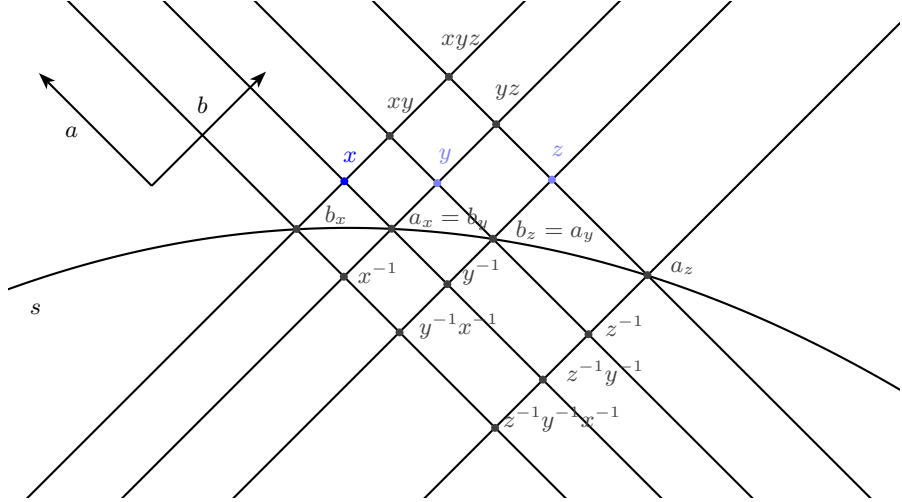
\begin{center} 
\newrgbcolor{xdxdff}{0.49 0.49 1}
\psset{xunit=0.5cm,yunit=0.5cm,algebraic=true,dotstyle=o,dotsize=3pt 0,linewidth=0.8pt,arrowsize=3pt 2,arrowinset=0.25}
\begin{pspicture*}(-4.86,-6.33)(18.83,6.85)
\psplot[plotpoints=200]{-9.857135596603504}{28.834528663528886}{0*(x-9.18)^3-0.02*(x-4.18)^2+0.84}
\psline{->}(-1.04,1.96)(2,5)
\psline{->}(-1.04,1.96)(-4.04,4.96)
\psplot{-9.86}{28.83}{(-6.08--3.04*x)/3.04}
\psplot{-9.86}{28.83}{(-18.48--3*x)/-3}
\psplot{-9.86}{28.83}{(-13.76--3.04*x)/3.04}
\psplot{-9.86}{28.83}{(-25.74--3*x)/-3}
\psplot{-9.86}{28.83}{(-22.75--3.04*x)/3.04}
\psplot{-9.86}{28.83}{(-35.16--3*x)/-3}
\psplot{-9.86}{28.83}{(-10.84--3*x)/-3}
\psplot{-9.86}{28.83}{(-38.2--3.04*x)/3.04}
\begin{scriptsize}
\rput[bl](-4.28,-1.44){$s$}
\rput[bl](0.17,3.91){$b$}
\rput[bl](-3.34,3.26){$a$}
\psdots[dotstyle=*,linecolor=blue](4.08,2.08)
\rput[bl](4.05,2.58){\blue{$x$}}
\psdots[dotstyle=*,linecolor=darkgray](-44.45,-46.45)
\psdots[dotstyle=*,linecolor=darkgray](2.81,0.81)
\rput[bl](3.56,0.94){\darkgray{$b_x$}}
\psdots[dotstyle=*,linecolor=darkgray](5.34,0.82)
\rput[bl](5.8,0.78){\darkgray{$a_x=b_y$}}
\psdots[dotstyle=*,linecolor=darkgray](53.02,-46.86)
\psdots[dotstyle=*,linecolor=xdxdff](6.55,2.03)
\rput[bl](6.59,2.54){\xdxdff{$y$}}
\psdots[dotstyle=*,linecolor=darkgray](8.03,0.55)
\rput[bl](8.64,0.45){\darkgray{$b_z=a_y$}}
\psdots[dotstyle=*,linecolor=darkgray](50.33,-41.75)
\psdots[dotstyle=*,linecolor=xdxdff](9.6,2.12)
\rput[bl](9.59,2.74){\xdxdff{$z$}}
\psdots[dotstyle=*,linecolor=darkgray](12.14,-0.42)
\rput[bl](12.75,-0.39){\darkgray{$a_z$}}
\psdots[dotstyle=*,linecolor=darkgray](46.22,-34.5)
\psdots[dotstyle=*,linecolor=darkgray](5.29,3.29)
\rput[bl](5.22,3.91){\darkgray{$xy$}}
\psdots[dotstyle=*,linecolor=darkgray](6.86,4.86)
\rput[bl](6.66,5.57){\darkgray{$xyz$}}
\psdots[dotstyle=*,linecolor=darkgray](8.12,3.6)
\rput[bl](8.09,4.24){\darkgray{$yz$}}
\psdots[dotstyle=*,linecolor=darkgray](10.57,-1.99)
\rput[bl](11.02,-2.05){\darkgray{$z\inv$}}
\psdots[dotstyle=*,linecolor=darkgray](9.36,-3.2)
\rput[bl](9.98,-3.35){\darkgray{$z\inv y\inv$}}
\psdots[dotstyle=*,linecolor=darkgray](8.09,-4.48)
\rput[bl](8.3,-4.59){\darkgray{$z\inv y\inv x\inv$}}
\psdots[dotstyle=*,linecolor=darkgray](5.55,-1.94)
\rput[bl](6.07,-2.11){\darkgray{$y\inv x\inv$}}
\psdots[dotstyle=*,linecolor=darkgray](6.82,-0.66)
\rput[bl](7.21,-0.68){\darkgray{$y\inv$}}
\psdots[dotstyle=*,linecolor=darkgray](4.07,-0.46)
\rput[bl](4.44,-0.62){\darkgray{$x\inv$}}
\end{scriptsize}
\end{pspicture*}
\end{center}

\nin
Summing up, just as for torsors and groups, working with a ternary map has the advantage of using a single map,
containing both inversion and the binary multiplications.
For instance, categorical definitions are greatly simplyfied by viewing groups and groupoids as defined by ternary
maps, and adding a unit element, respectively a bisection, as structure. 
Obviously, there is a forgetful functor forgetting this additional structure.
For general pregroupoids, Kock has constructed an adjoint functor to this forgetful functor, see \cite{Ko87, Ko05, Ko07}.

\subsection{On transitive pregroupoids}\label{ssec:trans}
It is well-known that general groupoids can be decomposed into a disjoint union  of transitive groupoids
(see \cite{CW, Ma05}). 
For pregroupoids, the analog is as follows.

\begin{definition}
Let $(M,a,b,[\quad])$ be a pregroupoid.
Recall (Lemma \ref{la:abcommute}) that the relations $a,b$ commute, and hence
$c:=ab$ is again an equivalence relation on $M$.
The quivalence classes of $c$ are called the {\em connected components} of $(M,a,b,[\quad])$.
We say that the pregroupoid is {\em transitive} if it has just one equivalence class. In other terms,
 the morphism onto the pair pregroupoid
$$
M \to M/a \times M/b, \quad x \mapsto ([x]_a,[x]_b)
$$
(see Example \ref{ex:pairpregroupoid})
is surjective.
\end{definition}

\nin
It is now obvious that every pregroupoid can be decomposed into a disjoint union of transitive pregroupoids.
Moreover, any transitive pregroupoid can be decomposed, in a non-canonical way, after
fixing a base point $([x]_a,[z]_b) \in M/a \times M/b$,
as a product
\begin{equation}
M \cong H \times (M/a \times M/b),
\end{equation}
where,
$H := [x]_a \cap [y]_b$ is  a torsor with product $xy\inv z$, considered as a pregroupoid over a point
(see \cite{Brown},\ p.119 for this item, in the case of groupoids).
Next, since the group $L:=\Bij(E) \times \Bij(F)$ acts transitively by automorphisms of the
pair pregroupoid on $E \times F$, by putting things together, it can be shown  that every transitive pregroupoid is
a quotient of a homogeneous pregroupoid $(G,G/A,B\backslash G)$ (example \ref{ex:homog!}; but of course the group $L$ 
is chosen much too big for everydaylife-situations).
Finally, if the pregroupoid does not admit bisections, by using Kock's construction, we may pass to an
even bigger transitive enveloping groupoid.



\end{document}